\documentclass[12pt,a4paper]{article}
\usepackage{amsmath,amsfonts,amssymb,amsthm}
\usepackage{graphicx}
\usepackage{fullpage}
\usepackage{boldline}
\usepackage{array,makecell,booktabs}

\usepackage[svgnames, table]{xcolor}

 \usepackage{subfig}
 \usepackage{wrapfig}
 \usepackage{enumerate}
\usepackage{color}

\makeatletter
\newcommand\footnoteref[1]{\protected@xdef\@thefnmark{\ref{#1}}\@footnotemark}
\makeatother

\usepackage{hyperref} 
\definecolor{modra3}{rgb}{.1,.0,.4}
\hypersetup{colorlinks=true, linkcolor=modra3, urlcolor=modra3, citecolor=modra3, pdfpagemode=UseNone, pdfstartview=} 
\if 10
\usepackage[inline]{showlabels} 
\definecolor{zelena}{rgb}{0,.35,0}

\showlabels{cite}
\showlabels{bibitem}
\fi

\if10     
\usepackage[mathlines]{lineno}
\newcommand*\patchAmsMathEnvironmentForLineno[1]{%
  \expandafter\let\csname old#1\expandafter\endcsname\csname #1\endcsname
  \expandafter\let\csname oldend#1\expandafter\endcsname\csname end#1\endcsname
  \renewenvironment{#1}%
     {\linenomath\csname old#1\endcsname}%
     {\csname oldend#1\endcsname\endlinenomath}}%
\newcommand*\patchBothAmsMathEnvironmentsForLineno[1]{%
  \patchAmsMathEnvironmentForLineno{#1}%
  \patchAmsMathEnvironmentForLineno{#1*}}%
\AtBeginDocument{%
\patchBothAmsMathEnvironmentsForLineno{equation}%
\patchBothAmsMathEnvironmentsForLineno{align}%
\patchBothAmsMathEnvironmentsForLineno{flalign}%
\patchBothAmsMathEnvironmentsForLineno{alignat}%
\patchBothAmsMathEnvironmentsForLineno{gather}%
\patchBothAmsMathEnvironmentsForLineno{multline}%
}
\linenumbers
\fi

\def\rank{\mathrm{rank}}

\definecolor{modra}{rgb}{0,0,.8}
\definecolor{piros}{rgb}{.8,0,0}
\def\marrow{{\marginpar[\hfill$\Rrightarrow$]{$\Lleftarrow$}}}
\if 10    
\def\jk#1{{\color{modra} {\sc JK: }{\marrow\sf #1} }}
\def\rf#1{{\color{piros} {\sc RF: }{\marrow\sf #1} }}
\else
\def\jk#1{}
\def\rf#1{}
\fi
\newif\iflong
\longtrue

\newtheorem{theorem}{Theorem} 
\newtheorem{lemma}[theorem]{Lemma}
\newtheorem{corollary}[theorem]{Corollary}
\newtheorem{proposition}[theorem]{Proposition}

\newtheorem{claim}[theorem]{Claim}

\theoremstyle{remark}
\newtheorem{remark}[theorem]{Remark}

\theoremstyle{definition}

\begin{document}

\title{$\mathbb{Z}_2$-genus of graphs and minimum rank of partial symmetric matrices}

\author{Radoslav Fulek\thanks{IST, Klosterneuburg, Austria; \texttt{radoslav.fulek@gmail.com}. The research leading to these results has received funding from the People Programme (Marie Curie Actions) of the European Union's Seventh Framework Programme (FP7/2007-2013) under REA grant agreement no [291734].}
  \and
  Jan Kyn\v{c}l\thanks{Department of Applied Mathematics, Charles University, Faculty of Mathematics and Physics,
    Malostransk\'e n\'am.~25, 118 00~ Praha 1, Czech Republic;
    \texttt{kyncl@kam.mff.cuni.cz}. Supported by project
    19-04113Y of the Czech Science Foundation (GA\v{C}R) and by Charles University project
    UNCE/SCI/004.}}



\maketitle

\begin{abstract}
  The \emph{genus} $\mathrm{g}(G)$ of a graph $G$ is the minimum $g$ such that $G$ has an embedding on the orientable surface $M_g$ of genus $g$.
  A drawing of a graph on a surface is \emph{independently even} if every pair of nonadjacent edges in the drawing crosses an even number of times. The \emph{$\mathbb{Z}_2$-genus} of a graph $G$, denoted by $\mathrm{g}_0(G)$, is the minimum $g$ such that $G$ has an independently even drawing on $M_g$.

  In 2013, Schaefer and \v{S}tefankovi\v{c} proved that the $\mathbb{Z}_2$-genus of a graph is additive over 2-connected blocks, which holds also for the genus by a result of Battle, Harary and Kodama from 1962.  Motivated by a problem  by Schaefer and \v{S}tefankovi\v{c} about an extension of their result to the setting of so-called 2-amalgamations
  we prove the following. If $G=G_1\cup G_2$, $G_1$ and $G_2$ intersect in two vertices $u$ and $v$, and $G\setminus\{u,v\}$ has $k$ connected components (among which we count the edge $uv$ if present), then $|\mathrm{g}_0(G)-(\mathrm{g}_0(G_1)+\mathrm{g}_0(G_2))|\le  k+1$.
  For complete bipartite graphs $K_{m,n}$, with $n\ge m\ge 3$, we prove that $\frac{\mathrm{g}_0(K_{m,n})}{\mathrm{g}(K_{m,n})}=1-O(\frac{1}{n})$. Similar results are proved also for the Euler genus.

  We express the $\mathbb{Z}_2$-genus of a graph using the minimum rank of partial symmetric matrices over $\mathbb{Z}_2$; a problem that might be of independent interest.
\end{abstract}

\section{Introduction}

The \emph{genus} $\mathrm{g}(G)$ of a graph $G$ is the minimum $g$ such that $G$ has an embedding on the orientable surface $M_g$ of genus $g$.
Similarly, the \emph{Euler genus} $\mathrm{eg}(G)$ of $G$ is the minimum $g$ such that $G$ has an embedding on a surface of Euler genus $g$.
We say that two edges in a graph are \emph{independent} (also \emph{nonadjacent}) if they do not share a vertex. The \emph{$\mathbb{Z}_2$-genus} $\mathrm{g}_0(G)$ and \emph{Euler $\mathbb{Z}_2$-genus} $\mathrm{eg}_0(G)$ of $G$ are defined as the minimum $g$ such that $G$ has a drawing on $M_g$ and a surface of Euler genus $g$, respectively, with every pair of independent edges crossing an even number of times. Clearly, $\mathrm{g}_0(G)\le \mathrm{g}(G)$ and $\mathrm{eg}_0(G)\le \mathrm{eg}(G)$.

The definition of the $\mathbb{Z}_2$-genus and Euler $\mathbb{Z}_2$-genus is motivated by the strong Hanani--Tutte theorem~\cite{Ha34_uber, Tutte70_toward} stating that a graph is planar if and only if its $\mathbb{Z}_2$-genus is 0.
Many variants and extensions of the theorem have been proved~\cite{CN00_thrackles,FK18_approx,PT00_crossings,Sch13b_towards,Sko03_approximability}, and they found various applications in combinatorial and computational geometry; see the survey by Schaefer~\cite{Sch13_hananitutte}.

 It had been a long-standing open problem whether the strong Hanani--Tutte theorem extends to surfaces other than the plane and projective plane, although the problem was first explicitly stated in print by Schaefer and {\v{S}}tefankovi{\v{c}}~\cite{SS13_block} in 2013. They conjectured that the strong Hanani--Tutte theorem extends to every orientable surface, that is, $\mathrm{g}_0(G)=\mathrm{g}(G)$ for every graph $G$. They proved that a minimal counterexample to their conjecture must be $2$-connected; this is just a restatement of their block additivity result, which we discuss later in this section. In a recent manuscript~\cite{FK17_genus4}, we provided an explicit construction of a graph $G$ for which $\mathrm{g}(G)=5$ and $\mathrm{g}_0(G)\le 4$,
thereby refuting the conjecture.
Nevertheless, the conjecture by Schaefer and {\v{S}}tefankovi{\v{c}}~\cite{SS13_block} that $\mathrm{eg}_0(G)=\mathrm{eg}(G)$ for every graph $G$ might still be true.

The conjecture has been verified only for graphs $G$ with $\mathrm{eg}(G)\le 1$: Pelsmajer, Schaefer and Stasi~\cite{PSS09_pp} proved that the strong Hanani--Tutte theorem extends to the projective plane, using the characterization of projective planar graphs by an explicit list of forbidden minors. Recently, Colin de Verdi\`{e}re et al.~\cite{CKPPT16_direct} gave a constructive proof of the same result.

Schaefer and {\v{S}}tefankovi{\v{c}}~\cite{SS13_block} also formulated a weaker form of their conjecture about the $\mathbb{Z}_2$-genus, stating that there exists a function $f:\mathbb{N}\rightarrow \mathbb{N}$ such that $\mathrm{g}(G)\le f(\mathrm{g}_0(G))$ for every graph $G$. Assuming the validity of an unpublished Ramsey-type result by Robertson and Seymour, the existence of such $f$ follows as a corollary from our recent result~\cite{FK18_z2} stating that $\mathrm{g}_0(G)=\mathrm{g}(G)$ for the graphs $G$ in the so-called family of Kuratowski minors. Regarding the asymptotics of $f$, we do not have any explicit upper bound on $f$, and
the existence of $G$ with $\mathrm{g}(G)=5$ and $\mathrm{g}_0(G)\le 4$ implies that $f(k)\ge 5k/4$~\cite[Corollary 11]{FK17_genus4}.

As the next step towards a good understanding of the relation between the (Euler) genus and the (Euler) $\mathbb{Z}_2$-genus we provide further indication of their similarity.
We will build upon techniques introduced in~\cite{SS13_block} and~\cite{FK18_z2}, and reduce the problem of  estimating  the (Euler) $\mathbb{Z}_2$-genus to the problem of estimating the minimum rank of partial symmetric matrices over~$\mathbb{Z}_2$.

First, we extend our recent result determining the $\mathbb{Z}_2$-genus of $K_{3,n}$~\cite[Proposition 18]{FK18_z2} in a weaker form to all
complete bipartite graphs.
A classical result by Ringel~\cite{Bo78_Kmn,Ri65_Kmn,R65_bipnonorient},~\cite[Theorem 4.4.7]{MT01_graphs},~\cite[Theorem 4.5.3]{GrTu01_theory} states that for $m,n\ge 2$, we have
$\mathrm{g}(K_{m,n})=\left\lceil \frac{(m-2)(n-2)}{4}\right\rceil$ and $\mathrm{eg}(K_{m,n})=\left\lceil\frac{(m-2)(n-2)}{2}\right\rceil$.

\begin{theorem}
\label{thm:Knm} 
If $n\ge m\ge 3$, then
\begin{align*}
\mathrm{g}_0(K_{m,n}) &\ge  \frac{(n-2)(m-2)}{4}-\frac{m-3}{2} \ \text{ and } \\
\mathrm{eg}_0(K_{m,n}) &\ge \frac{(n-2)(m-2)}{2}-(m-3).
\end{align*}
\end{theorem}

\jk{tady jsem mel pocit, ze umime polovicni chybovy clen... epizoda 1.5 "Clone Wars" - zaplaty, u kterych se da zvolit, jestli obsahuji i diagonalni blok nebo ne.}

Our second result is a $\mathbb{Z}_2$-variant of the results of
Stahl~\cite{Sta80_permutations}, Decker, Glover and Huneke~\cite{DGH81_2amalg,DGH85_2amalgamation}, Miller~\cite{Mi87_additivity} and
Richter~\cite{R87_2amalgamation} showing that the genus and Euler genus of graphs are almost additive over $2$-amalgamations, which we now describe in detail.

\begin{figure}
    \centering
    \includegraphics{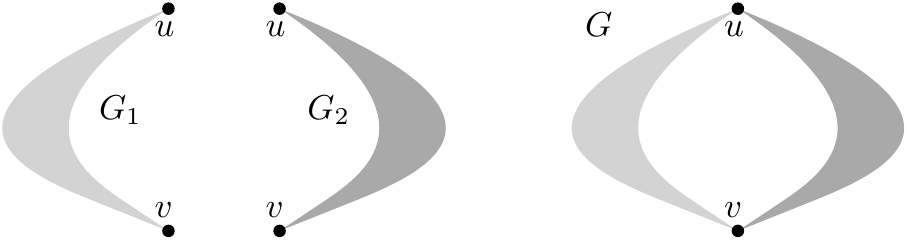}
    \caption{A 2-amalgamation $G=\Pi_{u,v}(G_1,G_2)$.}
    \label{fig:2-amalgamation}
\end{figure}

We say that a graph $G$ is a \emph{$k$-amalgamation} of graphs $G_1$ and $G_2$ (with respect to vertices $x_1,\ldots, x_k$) if $G=(V(G_1) \cup V(G_2),E(G_1) \cup E(G_2))$ and $V(G_1)\cap V(G_2)= \{x_1,\dots, x_k\}$, and we write $G=\amalg_{x_1,\ldots, x_k}(G_1,G_2)$. See Figure~\ref{fig:2-amalgamation} for an illustration.

An old result of Battle, Harary, Kodama and Youngs~\cite{BHKY62_additivity},~\cite[Theorem 4.4.2]{MT01_graphs},~\cite[Theorem 3.5.3]{GrTu01_theory} states that the genus of a graph is additive over its 2-connected blocks. In other words,
if $G$ is a $1$-amalgamation of $G_1$ and $G_2$ then
$\mathrm{g}(G) = \mathrm{g}(G_1)+\mathrm{g}(G_2)$. Stahl and Beinecke~\cite[Corollary 2]{StBe77_blocks},~\cite[Theorem 4.4.3]{MT01_graphs} and Miller~\cite[Theorem 1]{Mi87_additivity} proved that the same holds for the Euler genus, that is,
$\mathrm{eg}(G)=\mathrm{eg}(G_1)+\mathrm{eg}(G_2)$.
 Neither the genus nor the Euler genus are additive over $2$-amalgamations: for example, the nonplanar graph $K_5$ can be expressed as a $2$-amalgamation of two planar graphs in several ways. Nevertheless, the additivity in this case fails only by at most 1 for the genus and by at most 2 for the Euler genus. Formally,
Stahl~\cite{Sta80_permutations}
and Decker, Glover and Huneke~\cite{DGH81_2amalg,DGH85_2amalgamation} proved that
if $G$ is a 2-amalgamation of $G_1$ and $G_2$ then
$|\mathrm{g}(G)-(\mathrm{g}(G_1)+\mathrm{g}(G_2))| \le 1$. For the Euler genus, Miller~\cite{Mi87_additivity} proved its additivity over edge-amalgamations, which implies $\mathrm{eg}(G_1)+\mathrm{eg}(G_2) \le \mathrm{eg}(G)\le \mathrm{eg}(G_1)+\mathrm{eg}(G_2)+2$. Richter~\cite{R87_2amalgamation} later proved a more precise formula for the Euler genus of $2$-amalgamations with respect to a pair of nonadjacent vertices.


Schaefer and \v{S}tefankovi\v{c}~\cite{SS13_block} showed that the $\mathbb{Z}_2$-genus and Euler $\mathbb{Z}_2$-genus are additivite over $2$-connected blocks and they asked whether $|\mathrm{g}_0(G)-(\mathrm{g}_0(G_1)+\mathrm{g}_0(G_2))| \le 1$ if $G$ is a $2$-amalgamation of $G_1$ and $G_2$, as an analogue of the result by Stahl~\cite{Sta80_permutations} and Decker, Glover and Huneke~\cite{DGH81_2amalg,DGH85_2amalgamation}. We prove a slightly weaker variant of almost-additivity over $2$-amalgamations for both the $\mathbb{Z}_2$-genus and the Euler $\mathbb{Z}_2$-genus.

\begin{theorem}
  \label{thm:2amalg}
  Let $G$ be a $2$-amalgamation $\amalg_{v,u}(G_1,G_2)$. Let $l$ be the total number of connected components of $G-u-v$ in $G$. Let $k=l$ if $uv\not\in E(G)$ and $k=l+1$ if $uv\in E(G)$. Then
\[
\begin{array}{lr@{\hspace{5pt}}c@{\hspace{5pt}}c@{\hspace{5pt}}c@{\hspace{5pt}}l}
  \mathrm{a)}& \mathrm{g}_0(G_1)+\mathrm{g}_0(G_2)-(k+1) &\le &\mathrm{g}_0(G) &\le&\mathrm{g}_0(G_1)+\mathrm{g}_0(G_2)+1, \text{ and } \\
  \mathrm{b)}& \mathrm{eg}_0(G_1)+\mathrm{eg}_0(G_2)-(2k-1) &\le &\mathrm{eg}_0(G) &\le&\mathrm{eg}_0(G_1)+\mathrm{eg}_0(G_2)+2.  
\end{array}   
\]
\end{theorem}

\paragraph*{ Organization.} 
We give basic definitions and tools in Sections~\ref{sec:graphs} and~\ref{section_tools}.
In Section~\ref{sec:matrixProperties}, we present linear-algebraic results lying at the heart of our arguments.
In Section~\ref{sec:knm} and~\ref{sec:amalgamations}, we prove Theorem~\ref{thm:Knm} and Theorem~\ref{thm:2amalg}, respectively.
In Section~\ref{sec:amalgamations}, in order to illustrate our techniques in a simpler setting, we first reprove the block additivity result for the  Euler $\mathbb{Z}_2$-genus.
We finish with concluding remarks in Section~\ref{sec:conclusion}.



\section{Graphs on surfaces}

\label{sec:graphs}

We refer to the monograph by Mohar and Thomassen~\cite{MT01_graphs} for a detailed introduction into surfaces and graph embeddings.
By a {\em surface} we mean a connected compact $2$-dimensional topological manifold. Every surface is either {\em orientable} (has two sides) or {\em nonorientable} (has only one side). Every orientable surface $S$ is obtained from the sphere by attaching $g \ge 0$ \emph{handles}, and this number $g$ is called the {\em genus} of $S$. 
Similarly, every nonorientable surface $S$ is obtained from the sphere by attaching $g \ge 1$ \emph{crosscaps}, and this number $g$ is called the {\em (nonorientable) genus} of $S$. The simplest orientable surfaces are the sphere (with genus $0$) and the torus (with genus $1$). The simplest nonorientable surfaces are the projective plane (with genus $1$) and the Klein bottle (with genus $2$). We denote the orientable surface of genus $g$ by $M_g$, and the nonorientable surface of genus $g$ by $N_g$.
The \emph{Euler genus} of $M_g$ is $2g$ and the Euler genus of $N_g$ is $g$.

Let $G=(V,E)$ be a graph or a multigraph with no loops,
and let $S$ be a surface. 
A \emph{drawing} of $G$ on $S$ is a representation of $G$ where every vertex is represented by a unique point in $S$ and every
edge $e$ joining vertices $u$ and $v$ is represented by a simple curve in $S$ joining the two points that represent $u$ and $v$.
If it leads to no confusion, we do not distinguish between
a vertex or an edge and its representation in the drawing and we use the words ``vertex'' and ``edge'' in both contexts. We assume that in a drawing no edge passes through a vertex,
no two edges touch, every edge has only finitely many intersection points with other edges and no three edges cross at the same inner point. In particular, every common point of two edges is either their common endpoint or a crossing.
Let $\mathcal{D}$ be a drawing of a graph $G$.
We denote by $\mathrm{cr}_\mathcal{D}(e,f)$ the number of crossings between the edges $e$ and $f$ in $\mathcal{D}$.
A drawing of $G$ on $S$ is an \emph{embedding} if no two edges cross.

A drawing of a graph is \emph{independently even} if every pair of independent edges in the drawing crosses an even number of times.
In the literature, the notion of \emph{$\mathbb{Z}_2$-embedding} is used to denote an independently even drawing~\cite{SS13_block}, but also an \emph{even drawing}~\cite{CN00_thrackles} in which all pairs of edges cross evenly.




\section{Topological and algebraic tools}
\label{section_tools}
\jk{jeste treba doplnit nazev}

\subsection{Combinatorial representation of drawings}
\label{sec:combinatorial}

Schaefer and \v{S}tefankovi\v{c}~\cite{SS13_block} used the following combinatorial representation of drawings of graphs on $M_g$ and $N_g$. First, every drawing of a graph on $M_g$ can be considered as a drawing on the nonorientable surface $N_{2g+1}$, since $M_g$ minus a point is homeomorphic to an open subset of $N_{2g+1}$. The surface $N_{h}$ minus a point can be represented combinatorially as the plane with $h$ \emph{crosscaps}. A crosscap at a point $x$ is a combinatorial representation of a M\"obius strip whose boundary is identified with the boundary of a small circular hole centered in $x$. Informally, the main ``objective'' of a crosscap is to allow a set of curves intersect transversally at $x$ without counting it as a crossing.

Let $\mathcal{D}$ be a drawing of a graph $G$ in the plane with $h$ crosscaps.
To every edge $e\in E(G)$ we assign a vector $y^{\mathcal{D}}_e$ (or simply $y_e$) from $\mathbb{Z}_2^{h}$ such that $(y^{\mathcal{D}}_e)_i=1$ if and only if $e$ passes an odd number of times through the $i$th crosscap. 

\begin{figure}
    \centering
    \includegraphics{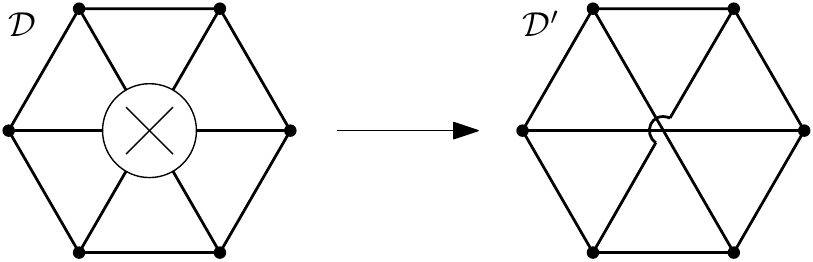}
    \caption{An embedding $\mathcal{D}$ of $K_{3,3}$ in the plane with a single crosscap (left) and its planarization $\mathcal{D}'$ (right).}
    \label{fig:planarization}
\end{figure}


\begin{figure}
    \centering
    \includegraphics{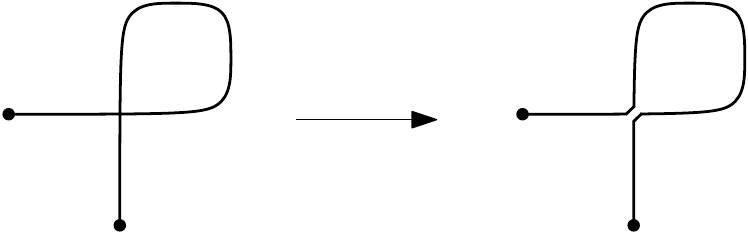}
    \caption{Removing a self-crossing of an edge.}
    \label{fig:self-crossing}
\end{figure}

Given a drawing $\mathcal{D}$ of a graph $G$ in the plane with $h$ crosscaps, the \emph{planarization} of
$\mathcal{D}$ is a  drawing $\mathcal{D}'$ of $G$ in the plane, obtained from $\mathcal{D}$
as follows; see Figure~\ref{fig:planarization} for an illustration. We turn the crosscaps into holes, fill the holes with discs,  reconnect the severed edges of $G$ by simple curves drawn across the filling discs while avoiding creating common crossing points of three and more edges, and finally we
 eliminate self-crossings of edges by cutting and rerouting the edges at such crossings; see Figure~\ref{fig:self-crossing}. 
Since $\mathcal{D}$ represents a drawing $\mathcal{D}_h$ on $N_h$, we denote by $\mathrm{cr}^*_{\mathcal{D}}(e,f)$ the number of crossings between the edges $e$ and $f$ that occur outside crosscaps in $\mathcal{D}$, which is equal to $\mathrm{cr}_{\mathcal{D}_h}(e,f)$. Writing $y_e^\top y_f$ for the scalar product of $y_e$ and $y_f$, we have
\begin{equation}
\label{eqn:rel}
\mathrm{cr}^*_\mathcal{D}(e,f) \equiv \mathrm{cr}_{\mathcal{D}'}(e,f) + y_e^\top y_f \ (\text{mod } 2),
\end{equation}
%
since $y_e^\top y_f$ has the same parity as the number of new crossings between $e$ and $f$ introduced during the construction of the planarization.
If $\mathcal{D}$ represents a drawing on $M_g$ (in the plane with $h=2g+1$ crosscaps), we say that $\mathcal{D}$ is \emph{orientable}. This is equivalent with every cycle passing through the crosscaps an even number of times.

We will use the following three lemmata by Schaefer and \v{S}tefankovi\v{c}~\cite{SS13_block}. 

\begin{lemma}[{\cite[Lemma 5]{SS13_block}}]
  \label{lemma_forest}
Let $G$ be a graph that has an independently even drawing $\mathcal{D}$ on a surface $S$ and let $F$ be a forest in $G$. Let $h=2g+1$ if $S=M_g$ and $h=g$ if $S=N_g$. Then $G$ has a drawing $\mathcal{E}$ in the plane with $h$ crosscaps, such that
  \begin{enumerate}
    \item[{\rm 1)}] for every pair of independent edges $e,f$ the number $\mathrm{cr}^*_\mathcal{E}(e,f)$ is even, and
    \item[{\rm 2)}] every edge $f$ of $F$ passes through each crosscap an even number of times; that is, $y^{\mathcal{E}}_f=0$.
  \end{enumerate}
\end{lemma}

We will be using Lemma~\ref{lemma_forest} when $G$ is connected and $F$ is a spanning tree of $G$.


\begin{lemma}[{\cite[Lemma 3]{SS13_block}}]
  \label{lemma_smallgenus}
Let $G$ be a graph that has an orientable drawing $\mathcal{D}$ in the plane with finitely many crosscaps such that for every pair of independent edges $e,f$ the number $\mathrm{cr}^*_\mathcal{D}(e,f)$ is even. Let $d$ be the dimension of the vector space generated by the set $\{y^{\mathcal{D}}_e$; $e\in E(G)\}$. Then $G$ has an independently even drawing on $M_{\lfloor d/2 \rfloor}$.
\end{lemma}

\begin{lemma}[{\cite[Lemma 4]{SS13_block}}]
  \label{lemma_smallgenus2}
  Let $G$ be a graph that has a drawing in the plane with finitely many crosscaps such that for every pair of independent edges $e,f$ the number $\mathrm{cr}^*_\mathcal{D}(e,f)$ is even. Let $d$ be the dimension of the vector space generated by the set $\{y^{\mathcal{D}}_e$; $e\in E(G)\}$. Then $G$ has an independently even drawing on a surface of Euler genus $d$.
\end{lemma}

\subsection{Bounding $\mathbb{Z}_2$-genus by matrix rank}

In Proposition~\ref{prop_strenghtening} we will prove a far-reaching consequence of Lemma~\ref{lemma_smallgenus} and Lemma~\ref{lemma_smallgenus2}. An immediate corollary of Proposition~\ref{prop_strenghtening} (Corollary~\ref{cor:rank} below) can be thought of as a $\mathbb{Z}_2$-variant of a result of Mohar~\cite[Theorem 3.1]{M89_obstruction}.

Roughly speaking, Proposition~\ref{prop_strenghtening} says that we can upper bound the $\mathbb{Z}_2$-genus and Euler $\mathbb{Z}_2$-genus of a graph $G$ in terms of  the rank of a symmetric matrix $A$ encoding the parity of crossings  between independent edges.
Then the entries in $A$ representing the parity of crossings between adjacent edges, and in the case of the Euler $\mathbb{Z}_2$-genus also diagonal elements,  can be chosen arbitrarily. The choice of such undetermined  entries minimizing
the rank of $A$ will play a crucial role in the proof of Theorem~\ref{thm:2amalg}.

We present some classical results that are crucial in the  proof of Proposition~\ref{prop_strenghtening}.
Following Albert~\cite{A38_symmetric}, a symmetric matrix over $\mathbb{Z}_2$ is \emph{alternate} if its diagonal contains only $0$-entries\footnote{Over an arbitrary field, an alternate matrix is a skew--symmetric matrix with only $0$-entries on the diagonal. The diagonal condition is redundant except for the fields of characteristic 2.}. It is a well-known fact that the rank of alternate matrices is even~\cite[Theorem 3]{A38_symmetric}. Two square matrices $A$ and $B$ over $\mathbb{Z}_2$  are \emph{congruent} if there exists an invertible matrix $C$  over $\mathbb{Z}_2$  such that $B=C^\top AC$. The following lemma combines
a pair of theorems from Albert~\cite{A38_symmetric}.
\jk{format citaci stackexchange je taky nejak standardizovany... minimalne by se melo psat jmeno tazatele nebo odpovidatele... Jmeno, nazev otazky, StackExchange, link (rok), pripadne zda je to link na neci odpoved/komentar apod.}
\jk{hlavne tady to je prave ten Arf invariant 0, coz je asi verohodnejsi reference nez nejaky post na stackechange... a "alternating matrix" jsem nikde jinde nevidel. .. tak update, nasel jsem "alternating bilinear form", a ze nonsingular quadratic form over Z2 has even dimension. Ale alternating form je mnohem vic nez jen nuly na uhlopricce - musi byt navic skew-symmetric.}
\jk{takze (v plne verzi) bude krasny uvod o bilinearnich/kvaratickych formach a Arf invariantu, a vybrane vysledky Alberta (Symmetric and alternating matrices in an arbitrary field) pripadne z MacWilliamse (Orthogonal matrices over finite fields). Faktorizace formy s Arf invariantem 1 jako lemma atd. Dukaz Proposition~\ref{prop_strenghtening} bude rozdelen na mensi lemmata, idealne oddelena topologicka a algebraicka cast.}
\jk{ale bylo by trapne, kdyby nas na Arf invariant musel upozornit recenzent... naopak by se dobre vyjimal v keywords ;) takovy spinavy trik jak ziskat pozornost snobu.}

\begin{lemma}[{\cite[Theorem 1 and 3]{A38_symmetric}}]
  \label{lemma:alternating}
  Every matrix congruent to an alternate matrix is an alternate matrix and
  two alternate matrices $A$ and $B$ over $\mathbb{Z}_2$ are congruent if and only if they have the same rank $2t$.
\end{lemma}

MacWilliams~\cite{M69_matrices} gave a concise exposition of the following result of Albert~\cite{A38_symmetric}.

\begin{lemma}[{\cite[Theorem 1]{M69_matrices}}]
  \label{lemma:factor}
  An invertible symmetric matrix $A$ over $\mathbb{Z}_2$ can be factored as $B^\top B$, where $B$ is invertible, if and only if $A$ is not alternate.
\end{lemma}

We extend Lemma~\ref{lemma:factor} to all the symmetric matrices.

\begin{corollary}
  \label{cor:factor}
  A non-alternate symmetric matrix $A$ over $\mathbb{Z}_2$  can be factored as $B^\top B$ such that $\mathrm{rank}(B)=\mathrm{rank}(A)$.
  An alternate symmetric matrix $A$ over $\mathbb{Z}_2$  can be factored as $B^\top B$ such that $\mathrm{rank}(A)\le \mathrm{rank}(B)\le \mathrm{rank}(A)+1$.
\end{corollary}

\begin{proof}
If $A$ is not alternate, let $A'=A$, otherwise 
let $A'$ be a symmetric matrix obtained from $A$ by adding a single dummy row and column. The added row and column have a single  $1$-entry, that is at their intersection.
We assume that the added row and column in $A'$ are the first row and column, respectively. Clearly, if $A$ is alternate, $\mathrm{rank}(A')=\mathrm{rank}(A)+1$.

 By applying symmetric elementary row and column operations to $A'$ we obtain a symmetric block matrix $D=\begin{pmatrix}
      D_{11} & 0 \\
      0      & 0
    \end{pmatrix}$, such that $D_{11}$ is invertible.
  Hence,  $A'=C^\top DC$, where $C$ is invertible, and $\mathrm{rank}(A')=\mathrm{rank}(D)$.
  Note that $D_{11}$ is symmetric.

Since $A'$ is not alternate,  by Lemma~\ref{lemma:alternating}, also $D_{11}$ is also not alternate, and by Lemma~\ref{lemma:factor}, we can factor $D_{11}$ as $E_{11}^\top E_{11}$,
  where $E_{11}$ is invertible and $\mathrm{rank}(E_{11})=\mathrm{rank}(D_{11})$. Let $E=
    \begin{pmatrix}
      E_{11} & 0 \\
      0      & 0
    \end{pmatrix}$.  We have $\mathrm{rank}(A')=\mathrm{rank}(EC)$, and $C^\top D C=C^\top (E^\top E)C=(EC)^\top EC=A'$.
    
    Let $B=EC$ if $A$ is not alternate, and let $B$ be obtained from $EC$ by deleting the first column if $A$ is alternate.
In either case $B^\top B=A$ by the definition of $A'$. Thus, noting that $\mathrm{rank}(A)=\mathrm{rank}(A')=\mathrm{rank}(B)$ if $A$ is not alternate, and that $\mathrm{rank}(A)+1=\mathrm{rank}(A')\ge \mathrm{rank}(B)\ge \mathrm{rank}(A')-1$, if $A$ is alternate, finishes the proof.
\end{proof}

\jk{asi uplne nerozumim ucelu nasledujiciho lemma - pouziva se na Kmn, amalgamace, nebo oboji? "choice of undetermined entries" .. nebo to je presne pro sousedni dvojice fundamentalnich hran? to by se kdyztak mohlo rict na rovinu. Nejak si nevzpominam, ze bychom neco tak obecneho pouzivali. }

\jk{definovat "matici nakresleni" nebo neco takoveho..? (ta by byla urcena cela). Zneni lemma na 6 radku je totiz celkem dost (i kdyz neni jedine takove). Zalezi asi jak moc chceme tu matici modifikovat; zatim uplne nemam prehled.}

Let $\mathcal{D}$ be a drawing of a graph $G$ in the plane.
Let $E(G)=\{e_1,\ldots, e_k\}$. A symmetric $k\times k$ matrix $A=(a_{ij})$ over $\mathbb{Z}_2$ \emph{represents}  $\mathcal{D}$ if for every independent pair $e_i$ and $e_j$ we have that   $a_{ij}=\mathrm{cr}_\mathcal{D}(e_i,e_j) \ \mathrm{mod}\ 2$.


\begin{proposition}
  \label{prop_strenghtening}
  Let $\mathcal{D}$ be a drawing of a graph $G$ in the plane.
  If a matrix $A$ represents $\mathcal{D}$ then $\mathrm{eg}_0(G)\le \mathrm{rank}(A)$. If additionally  $A$ has only zeros on the diagonal then  $\mathrm{g}_0(G)\le \mathrm{rank}(A)/2$.
\end{proposition}

\begin{figure}
  \centering
  \includegraphics{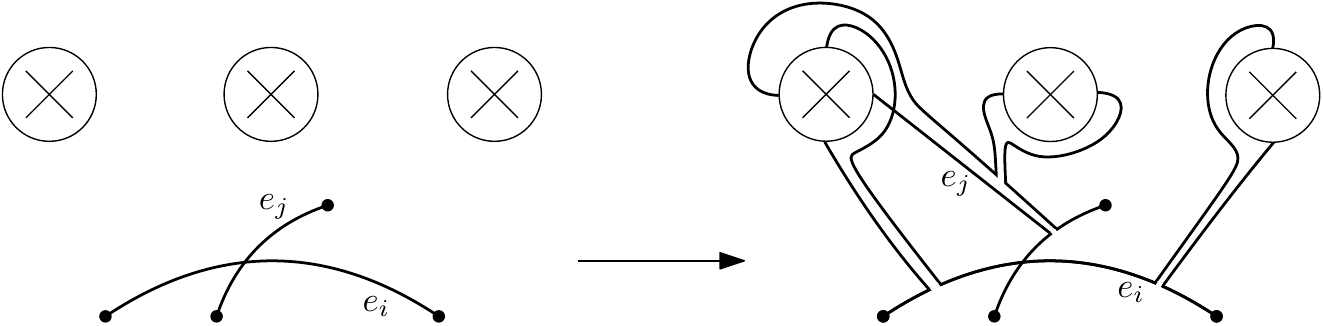}
  \caption{Pushing the edge $e_i$ and $e_j$ with the crosscap vector $y_{e_i}^\top=(1,0,1)$ and $y_{e_j}^\top = (1,1,0)$, respectively, over crosscaps.}
  \label{fig:pushOverCrosscap}
\end{figure}

\begin{proof}
If $\mathrm{rank}(A)=0$, then $\mathcal{D}$ is independently even, and 
hence $\mathrm{eg}_0(G)=\mathrm{g}_0(G)=0$. Now assume that 
$\mathrm{rank}(A)>0$.
 By Corollary~\ref{cor:factor}, we factor $A$ as $B^\top B$ where $B=
\begin{pmatrix}
       y_{e_1} & \ldots & y_{e_k}
\end{pmatrix}$
is an $h\times k$  matrix over $\mathbb{Z}_2$ such that 
$\mathrm{rank}(A)\le \mathrm{rank}(B) \le\mathrm{rank}(A)+1$. We will do the following. We 
interpret $y_{e_i}\in \mathbb{Z}_2^h$ as a crosscap vector of $e_i$, for 
each $i\in [k]$. Then we construct a drawing $\mathcal{D}_0$ of $G$ in the 
plane with $h$ crosscaps in which $\mathrm{cr}^*_{\mathcal{D}_0}(e,f)=0$ 
for every pair of independent edges $e,f$, and such that 
$y_{e_i}^{\mathcal{D}_0}=y_{e_i}$ for each $i\in [k]$.
  
In more detail, in the complement of $\mathcal{D}$ in the plane, we 
introduce $h$ crosscaps. For every $i\in [k]$ and $j\in [h]$, if 
$(y_{e_i})_j=1$, we  pull $e_i$ over the $j$th crosscap (in an arbitrary 
order). Since every edge is pulled over each crosscap at most once, we can 
easily avoid creating self-crossings of the edges. See~\cite[Fig. 
1]{SS13_block} or Figure~\ref{fig:pushOverCrosscap} for an illustration.
Let $\mathcal{D}_0$ be the resulting drawing in the plane with $h$ 
crosscaps. For every $e \in E(G)$, the parity of 
$\mathrm{cr}^*_{\mathcal{D}_0}(e_i,e)$ differs from the parity of 
$\mathrm{cr}_{\mathcal{D}}(e_i,e)$ if and only if $y_{e}^{\top}y_{e_i}$ is 
odd.

  We apply  Lemma~\ref{lemma_smallgenus2}  to conclude that $G$ has an independently even drawing on the surface of Euler genus $\mathrm{rank}(B)$. 
  Note that either $\mathrm{rank}(B)=\mathrm{rank}(A)$ or $\mathrm{rank}(B)=\mathrm{rank}(A)+1$.
  In the former, the Euler $\mathbb{Z}_2$-genus part of the proposition follows.
  In the latter, $A$ is alternate, that is,  $A$ has only $0$-entries on the diagonal, and hence, $\mathrm{rank}(A)$ is even by Lemma~\ref{lemma:alternating}.
  It also follows that $y_{e_i}^\top y_{e_i} \mod 2=0$ for all $i\in [k]$. 
  Thus, by Lemma~\ref{lemma_smallgenus} we have that $\mathrm{eg}_0(G)/2 \le\mathrm{g}_0(G)\le \lfloor \mathrm{rank}(B)/2 \rfloor =\mathrm{rank}(A)/2$, which also gives the $\mathbb{Z}_2$-genus part of the proposition.
   \end{proof}

An almost immediate corollary of Proposition~\ref{cor:rank2} is the following.

\begin{corollary}
\label{cor:rank2}
We have $\mathrm{eg}_0(G)=\min_{A,\mathcal{D}} \mathrm{rank}(A)$, where we minimize over symmetric matrices $A$ representing a drawing $\mathcal{D}$ of $G$ in the plane,
and $\mathrm{g}_0(G)=\min_{A,\mathcal{D}} \mathrm{rank}(A)/2$,
where we minimize over alternate matrices $A$ representing a drawing $\mathcal{D}$ of $G$ in the plane.
\end{corollary}

We state and prove a slightly more technical version of Corollary~\ref{cor:rank2} below that better fits our applications.

Let $\mathcal{D}$ be a drawing of $G$ in the plane.
Let $A$ be a matrix representing $\mathcal{D}$.
A symmetric submatrix $A'$ of $A$  \emph{essentially represents}
 $\mathcal{D}$ if $A'$ is obtained from $A$ as a restriction to the rows and columns corresponding to  $E'\subseteq E(G)$, where $E'$ is  such that every pair of independent edges $e\in E(G)\setminus E'$ and $f\in E(G)$ cross an even number of times in $\mathcal{D}$. The introduction of $E'$ serves the purpose of eliminating from $A$ rows and columns that could be all-zero. In the light of Lemma~\ref{lemma_forest}, the set of edges $E(G)\setminus E'$ can be therefore thought of forming a spanning forest of $G$.

\begin{corollary}
  \label{cor:rank}
  We have $\mathrm{eg}_0(G)=\min_{A,\mathcal{D}} \mathrm{rank}(A)$, where we minimize over symmetric matrices $A$ essentially representing a drawing $\mathcal{D}$ of $G$ in the plane.
  Similarly, $\mathrm{g}_0(G)=\min_{A,\mathcal{D}} \mathrm{rank}(A)/2$,
  where we minimize over alternate matrices $A$ essentially representing a drawing $\mathcal{D}$ of $G$ in the plane.
\end{corollary}

\begin{proof}
Given a drawing $\mathcal{D}$ of $G$ there exists a matrix representing $A$ such that the columns and rows corresponding to $E'$, such that every pair of independent edges $e\in E(G)\setminus E'$ and $f\in E(G)$ cross an even number of times in $\mathcal{D}$, are all-zero. Thus, without loss of generality, we can assume that a matrix $A$ representing $\mathcal{D}$ with the smallest rank is of this type.
Hence, in the following we assume that $\mathrm{rank}(A)=\mathrm{rank}(A')$, where $A'$ is the restriction of $A$ to the rows and columns corresponding to $E'$.

  By Proposition~\ref{prop_strenghtening}, it is enough to prove that $\mathrm{eg}_0(G)\ge \min_{A,\mathcal{D}} \mathrm{rank}(A)$ and
  $\mathrm{g}_0(G)\ge \min_{A,\mathcal{D}} \mathrm{rank}(A)/2$. Let $\mathcal{D}$ be a drawing   of a graph $G$ in the plane with  finitely many crosscaps witnessing the Euler $\mathbb{Z}_2$-genus or  $\mathbb{Z}_2$-genus of $G$. Let $B$ be the matrix whose columns are crosscap vectors of the edges $e\in E(G)$ associated with $\mathcal{D}$.
  Clearly, $A=B^\top B$ is a symmetric matrix that represents the planarization of $\mathcal{D}$, and $\mathrm{eg}_0(G)\ge\mathrm{rank}(B)\ge \mathrm{rank}(A)$ or $\mathrm{g}_0(G)\ge\mathrm{rank}(B)/2\ge \mathrm{rank}(A)/2$.
\end{proof}



\section{Minimum rank of partial symmetric matrices}
\label{sec:matrixProperties}
In this section we prove linear-algebraic results that we use to establish Theorem~\ref{thm:Knm} and Theorem~\ref{thm:2amalg}.
As usual, we let $I_n,{\bf 0}_n$ and $J_n$ denote the $n\times n$ identity matrix, all-zeros and all-ones matrix, respectively.
Let $A$ be an $m\times n$ matrix.
For $i_1,i_2\in [m]$ and $j_1,j_2 \in [n]$, we denote by  $A[i_1,\ldots,i_2;j_1,\ldots,j_2]$ the submatrix of $A$ obtained as the restriction of $A$ to the  $i_1,\ldots,i_2$th rows and
$j_1,\ldots,j_2$th columns.
For $I \subseteq [m]$ and $J\subseteq [n]$, we denote by
$A[I,J]$ the submatrix of $A$ obtained as the restriction of $A$
to the rows with indices in $I$ and the columns with indices in $J$.

\subsection{Tournament matrices}
\label{sec:tournament}

The aim of this subsection is to extend the result of de Caen~\cite{Ca91_tournament}, that shows a lower bound on the rank of tournament matrices, to certain block matrices.
This extension lies at the heart of the proof of  Theorem~\ref{thm:Knm}.

It is a well-known fact that the matrix rank is subadditive over an arbitrary field, that is, for any two matrices $A_1$ and $A_2$ of the same dimensions we have

\begin{equation}
  \label{eqn:subAdditive}
  \rank(A_1+A_2)\le \rank(A_1)+\rank(A_2).
\end{equation}

An $n\times n$ matrix $A=(a_{ij})$ over $\mathbb{Z}_2$ is a \emph{tournament matrix} if $a_{ij}=a_{ji}+1 \mod 2$ whenever $i\not=j$.
If $A$ is an $n\times n$ tournament matrix, then by a result of de Caen~\cite{Ca91_tournament}
\begin{equation}
  \label{eqn:tournamentDeCaen}
  \mathrm{rank}(A)\ge \left\lceil\frac{n-1}{2} \right\rceil.
\end{equation}
Indeed, $A+A^{\top}=J_{n}+I_{n}$ and $n-1\le \rank(I_n+J_n)$.
Thus, by~(\ref{eqn:subAdditive}) we obtain that $n-1\le \rank(A)+\rank(A^\top)=2\rank(A)$.

Let $A=(A_{ij})$ be an $m\times m$ block matrix, where each block
$A_{ij}$ is an $n\times n$ matrix over  $\mathbb{Z}_2$. Let $B$ be an $n\times n$ tournament  matrix.
Suppose that $A$ is symmetric and that each off-diagonal block $A_{ij}$, $i\not=j$ is either $B+D_{ij}$ or $J_n+B+D_{ij}$,
where $D_{ij}$ is a diagonal matrix over $\mathbb{Z}_2$.

\begin{lemma}
  \label{lemma:MatrixRank}
  If $m\ge 2$ then $\rank(A)\ge \left\lceil\frac{(m-1)(n-1)}{2}\right\rceil-(m-2)$.
\end{lemma}

\begin{proof}
  We perform a sequence of operations on $A$ that turn $A$ successively into block matrices $A_0, A_1, A_2, A_3$ and  $A_4$, where $A_l=(A_{ij}^l)$. The desired lower bound on the rank of $A$ is obtained from the tight lower bound on the rank of $A_4$ by tracing back the changes applied to $A$ in order to get $A_4$.

  Let $A_0$ be obtained from $A$ by adding the row $\begin{pmatrix}
      A_{11} & \ldots & A_{1m}
    \end{pmatrix}$ to every other row.
  Let $A_1$ be obtained from $A_0$ by discarding the first row of blocks
  $\begin{pmatrix}
      A_{11}^0 & \ldots & A_{1m}^0
    \end{pmatrix}$ and the first column of blocks
  $\begin{pmatrix}
      A_{11}^{0\top} & \ldots & A_{m1}^{0\top}
    \end{pmatrix}^\top$.
  Note that the diagonal blocks of $A_1$  are tournament matrices, and each off-diagonal block is either a  diagonal matrix, or has $1$-entries everywhere outside the diagonal. However, we would like to have only diagonal matrices as blocks below the diagonal blocks.

  Let $A=(A_{ij})$ be an $(m-1)\times (m-1)$ block matrix, where  $A_{ij}=J_n$, if $i>j$ and $A_{ij}^1$ is not a diagonal matrix; and $A_{ij}={\bf 0}_n$, otherwise. Let $A_2=A_1+A$. Hence, every below diagonal block $A_{ij}^2$ in $A_2$ is  a diagonal matrix.

  In what follows, we construct $A_3$ by performing elementary row and column  operations  on $A_2$  so that $A_3$ has the three following properties.

  \begin{enumerate}[(1)]
    \item
          \label{it:reqProp1}
          Every $1$-entry in a below diagonal block is the unique $1$-entry in its row and column.
    \item
          \label{it:reqProp2}
          Every diagonal block is a  tournament matrix.
    \item
          \label{it:reqProp3}
          Every off-diagonal block, that is below the diagonal, is a diagonal matrix.
  \end{enumerate}

  We proceed in steps indexed by $k$ starting from $k=1$. In the $k$th step, we construct an auxiliary matrix $C_k=(c_{ij}^k)$. We obtain $A_3$ as  $C_l$,  for some $l\ge 1$, satisfying (\ref{it:reqProp1})--(\ref{it:reqProp3}).
  We put $C_1=A_2$.
  During the $k$th step, we consider  $1$-entry $c_{i_kj_k}^k$ of $A_k$ violating~(\ref{it:reqProp1}) such that the submatrix $C_k[i_k,\ldots, (m-1)n; 1,\ldots ,j_k]$ has no $1$-entry violating the required property except for $c_{i_kj_k}^k$.   (If such a $1$-entry does not exist, we are done.)
  In particular, $C_k[i_k,\ldots, (m-1)n; 1,\ldots ,j_k]$ contains no other $1$-entry in the first row and the last column. We choose $c_{i_kj_k}^k$ maximizing $i_k$.  In $C_k$, we add the $i_k$th row to all the other rows containing a $1$-entry  in the $j_k$th column. Furthermore, we add the $j_k$th column to all the columns containing a $1$-entry  in the $i_k$th row. Let $C_{k+1}$ be the resulting matrix. Note that in $C_{k+1}$, $c_{i_kj_k}^{k+1}$ is no longer a $1$-entry violating~(\ref{it:reqProp1}).  Since every off-diagonal block of $C_{k}$ is a diagonal matrix, both (\ref{it:reqProp2}) and~(\ref{it:reqProp3}) still hold for $C_{k+1}$.

  \begin{claim}
    \label{claim:finiteness}
    If $i_{k+1}$ is defined then $i_{k+1}<i_{k}$.
  \end{claim}
  \begin{proof}
    For the sake of contradiction we assume that $i_{k+1}\ge i_{k}$.

    Since $C_k[i_k,\ldots, (m-1)n; 1,\ldots ,j_k]$ contains no other $1$-entry in the first row and the last column,  $c_{i_{k+1}j_{k+1}}^k$ must be a $1$-entry in $C_{k}$ violating~(\ref{it:reqProp1}), and therefore also the choice of $c_{i_kj_k}$ (contradiction).
  \end{proof}

  By Claim~\ref{claim:finiteness}, after finitely many steps we obtain the matrix $A_3$ in which each row and column contains at most one 1 in below diagonal blocks.

  Let $I$ and $J$ be the set of indices of the rows and columns, respectively, of $A_3$ with a single $1$-entry in the below diagonal block. Let $A_3'=A_3[I, J]$ and let $A_3''=A[[(m-1)n]\setminus (I \cup J),[(m-1)n]\setminus (I \cup J)]$.
  Let $A_4=\begin{pmatrix}
      A_{3}' & 0     \\
      0      & A_3''
    \end{pmatrix}$. We further refine the block structure of the matrix $A_4$ according to the blocks of $A_3''$ inherited from $A_3$. Thus, $A_3''$ still satisfies~(\ref{it:reqProp1})--(\ref{it:reqProp3}), and moreover, all the below diagonal blocks are zero matrices.

  By~(\ref{it:reqProp1}), we obtain that (i) the rank of $A_{11}^4$ is exactly the number of $1$-entries in the below diagonal blocks of $A_3$, and  every other block in the row and column of $A_{11}^4$ is a zero matrix. By~(\ref{it:reqProp2}), we obtain that  (ii)  the diagonal blocks of $A_4$ are tournament matrices except for $A_{11}^4$.
  Hence, by~(i),~(ii), the subadditivity of $\lceil.\rceil$, and~(\ref{eqn:tournamentDeCaen}) we have
  \begin{equation}
    \label{eqn:lowerBound}
    \left\lceil\frac{(m-1)(n-1)}{2}\right\rceil =\left\lceil\frac{(m-1)(n-1)-2\rank(A_{11}^4)}{2}\right\rceil+\rank(A_{11}^4)\le \rank(A_4)
  \end{equation}

  Recall that we obtained $A_4$ from $A_0$ by deleting rows and columns, and by performing elementary row and column operations. Thus, $\rank(A_4)\le \rank(A_0)$, and  by~(\ref{eqn:lowerBound}) we have $ \left\lceil\frac{(m-1)(n-1)}{2}\right\rceil  \le \rank(A_0)$. By the subadditivity of the matrix rank~(\ref{eqn:subAdditive}), $\rank(A_0)\le \rank(M)+\rank(A)$. Hence,
  $$ \left\lceil\frac{(m-1)(n-1)}{2}\right\rceil-(n-2)\le \left\lceil\frac{(m-1)(n-1)}{2}\right\rceil-\rank(A)\le \rank(A).$$
\end{proof}

\subsection{Block symmetric  matrices}
\label{sec:block}

In this section, we prove some minimum rank formulas for certain partial block symmetric matrices, that play an important role in the proof of Theorem~\ref{thm:2amalg}.
The study of the rank of partial block matrices  was initiated
in the work of Cohen et al.~\cite{CJRW89_rank}, Davis~\cite{D88_completing}, and Woerdeman~\cite{W87_lower}.
We present some results extending the previous work to the setting of symmetric matrices.
Let $A(X)=
  \begin{pmatrix}
    A_{11} & A_{12} \\
    A_{21} & X
  \end{pmatrix}$ be a  block matrix over an arbitrary field in which
the block $X$ is treated as a variable.
Woerdeman~\cite{W87_lower} and Davis~\cite{D88_completing} proved that
\begin{equation}
  \nonumber
  \min_{X} \mathrm{rank}(A(X))= \mathrm{rank}
  \begin{pmatrix}
    A_{11} & A_{12}
  \end{pmatrix}
  +\mathrm{rank} \begin{pmatrix}
    A_{11} \\ A_{12}
  \end{pmatrix}
  -\mathrm{rank}(A_{11}).
\end{equation}

The following lemma  says that the lower bound on
$\min_{X} \mathrm{rank}(A(X))$, where we minimize over symmetric $A(X)$, that is implied by the result of Davis and Woerdeman, is tight.

\begin{lemma}
  Let $A_{21}=A_{12}^\top$ and let $A_{11}$ be symmetric. Then
  \label{lemma:rankMin1}
  \begin{equation}
    \nonumber
    \min_{X} \mathrm{rank}(A(X))= 2\mathrm{rank}
    \begin{pmatrix}
      A_{11} & A_{12}
    \end{pmatrix}
    -\mathrm{rank}(A_{11}),
  \end{equation}
  where we minimize over symmetric $X$.
\end{lemma}
\begin{proof}
  We proceed as follows. We apply symmetric elementary row and column operations to $A(X)$ and refine its block structure thereby obtaining a 3 by 3 block matrix $B(Y)$, that has only a single block $Y$ whose entries depend  on $X$.
  In $B(Y)$, all the above anti-diagonal blocks will be zero matrices and  $Y$ will be below anti-diagonal
  as the only non-zero block.
  The lemma will follow by observing that $\mathrm{rank}(B(Y))$ is lower bounded by the sum of its anti-diagonal blocks and this lower bound is achieved if $Y$ is a zero matrix

  By applying symmetric elementary row and column operations to $A(X)$ we obtain
  a 3 by 3 symmetric block matrix $B(Y)=(B_{ij})$,  where  the submatrix
  $\begin{pmatrix}
      B_{11} & B_{12} \\
      B_{21} & B_{22}
    \end{pmatrix}$
  is obtained from $A_{11}$, the submatrix
  $\begin{pmatrix}
      B_{13} \\
      B_{23}
    \end{pmatrix}$
  from $A_{12}$,  the submatrix
  $\begin{pmatrix}
      B_{31} & B_{32}
    \end{pmatrix}$ from $A_{21}$,
  and the block $Y$ from $X$; with the following properties.
  \begin{enumerate}[(i)]
    \item The block $Y$ is the only one  depending on $X$.
    \item  The blocks $B_{11},B_{12},B_{21},B_{23}$ and $B_{32}$ are zero matrices.
    \item  $\mathrm{rank}(A_{11})=\mathrm{rank}(B_{22}), \ \mathrm{rank}(B_{13})=\mathrm{rank}(B_{31})=\mathrm{rank}
            \begin{pmatrix}
              A_{11} & A_{12}
            \end{pmatrix}
            - \ \mathrm{rank}(A_{11})$

  \end{enumerate}
  \begin{figure}
    \centering
    $A(X)=\left(
      \begin{array}{c|c}
          A_{11} & A_{12} \\
          \hline
          A_{21} & X
        \end{array}\right)\sim
      \left(\begin{array}{cc|c}
          0      & 0      & B_{13} \\
          0      & B_{22} & B_{23} \\ \hline
          B_{31} & B_{32} & X
        \end{array}\right)\sim
      \left(\begin{array}{cc|c}
          0      & 0      & B_{13} \\
          0      & B_{22} & 0      \\ \hline
          B_{31} & 0      & Y
        \end{array}\right)=B(Y)$
    \caption{Application of the elementary operations $A(X)$ to obtain  $B(Y)$.}
    \label{fig:minrank1}
  \end{figure}

  The matrix $B$ is obtained in two steps, see Figure~\ref{fig:minrank1}. In the first step, we apply elementary row and column operations
  to
  $\begin{pmatrix} A_{11} & A_{12}
    \end{pmatrix}$ and
  $\begin{pmatrix}
      A_{11} \\
      A_{21}
    \end{pmatrix}$,
  respectively, in order to turn $A_{11}$ into a 2 by 2 block matrix
  $\begin{pmatrix}
      0 & 0      \\
      0 & B_{22}
    \end{pmatrix}$,
  where $B_{22}$ is invertible and the remaining 3  blocks are zero matrices. Note that $\mathrm{rank}(B_{22})=\mathrm{rank}(A_{11})$.
  (The blocks that are zero matrices might be 0 by 0 matrices.)
  In the second step, we  symmetrically add rows and columns of $B$ passing through $B_{22}$ to  rows and columns, respectively, passing through  $X$ so that the rows and columns of $B_{22}$ have only $0$-entries outside of $B_{22}$. This is possible since $B_{22}$ is invertible.
  The modified blocks $A_{12}$ and $A_{21}$ are then refined accordingly into
  $\begin{pmatrix}
      B_{13} \\
      B_{23}
    \end{pmatrix}$ and
  $\begin{pmatrix}
      B_{31} & B_{32}
    \end{pmatrix}$, respectively. The block $X$ is changed into  $Y=B_{33}$.

  In both steps, we did not touch rows and columns of $A(X)$ passing through $X$ except adding to them a row or a column not passing through $X$, respectively, which verifies~(i).
  The second step can potentially change $X$,
  but does not affect the rows and columns passing through $B_{11}$, since $B_{12}$ and $B_{21}$ are zero matrices. Thus, (ii)  holds as $B_{11}$ is also a zero matrix, and $B_{23}$ and $B_{32}$ were cleared of 1s during the second step. Therefore $\mathrm{rank}(B_{22})+\mathrm{rank}(B_{13})=\mathrm{rank}\begin{pmatrix}
      A_{11} & A_{12}
    \end{pmatrix}$, which by symmetry verifies~(iii) if combined with the previously observed fact that $\mathrm{rank}(B_{22})=\mathrm{rank}(A_{11})$.

  We have $B(Y)=CA(X) C^{\top}$, where $C$ is invertible. Hence,
  by~(i) we have  $$\min_{X} \mathrm{rank}(A(X)) = \min_{Y} \mathrm{rank}(B(Y)).$$
  Thus, in order to prove the lemma it is enough to prove that  $$\min_{Y} \mathrm{rank}(B(Y))= 2\mathrm{rank}
    \begin{pmatrix}
      A_{11} & A_{12}
    \end{pmatrix}-\mathrm{rank}(A_{11}).$$
  By~(ii), $\mathrm{rank}(B_{22})+\mathrm{rank}(B_{13})+\mathrm{rank}(B_{31}) = \mathrm{rank}(B({0})) \le \mathrm{rank}(B(Y))$ for any $Y$. Hence,  $\min_{Y} \mathrm{rank}(B(Y))=\mathrm{rank}(B_{22})+\mathrm{rank}(B_{13})+\mathrm{rank}(B_{31})$. Noting that by~(iii) the right hand side in the previous equality is equal to $2\mathrm{rank}\begin{pmatrix}
      A_{11} & A_{12}
    \end{pmatrix}-\mathrm{rank}(A_{11})$   concludes the proof.
\end{proof}

The following lemma is used in the proof of Lemma~\ref{lemma:rankMin2}
which we prove right after.

\begin{lemma}
  \label{lemma:rankMin15}
  Let
  \[
    A(X_1,X_2)=
    \begin{pmatrix}
      X_1    & A_{12} \\
      A_{12} & X_2
    \end{pmatrix}
  \] be a  block matrix, such that $A_{21}=A_{12}^\top$, over an arbitrary field in which
  the blocks $X_1$ and $X_2$ are treated as  variables.
  Then
  \begin{equation}
    \nonumber
    \min_{X_1,X_2} \mathrm{rank}(A(X_1,X_2))=\mathrm{rank}(A_{12}),
  \end{equation}
  where we minimize over symmetric matrices $X_1$ and $X_2$.
\end{lemma}

\begin{proof}
  We apply symmetric elementary row and column operations to $A(X_1,X_2)$ thereby obtaining a 2 by 2 symmetric block matrix $B(Y_1,Y_2)=\begin{pmatrix}
      Y_1    & B_{12} \\
      B_{21} & Y_2
    \end{pmatrix}$ such that
  $B_{12}$ and $B_{21}$ is obtained from $A_{12}$ and $A_{21}$, respectively, with the following properties.
  \begin{enumerate}[(i)]
    \item Only the block $Y_1$ and $Y_2$ depends on the block $X_1$ and $X_2$, respectively.
    \item The blocks $B_{12}$ and $B_{21}$ are symmetric matrices and $B_{12}=B_{21}$.
  \end{enumerate}

  To this end we apply the  symmetric elementary row and column operations to $A(X_1,X_2)$ in order to transform $A_{12}$ and  $A_{21}$ into diagonal matrices.

  We have  that $B(Y_1,Y_2)=CA(X_1,X_2) C^{\top}$, where $C$ is invertible. Hence,
  by~(i) we have  $$\min_{X_1,X_2} \mathrm{rank}(A(X_1,X_2)) = \min_{Y_1,Y_2} \mathrm{rank}(B(Y_1,Y_2)).$$
  Thus, in order to prove the lemma it is enough to prove that  $$\min_{Y_1,Y_2} \mathrm{rank}(B(Y_1,Y_2))= \mathrm{rank}
    \begin{pmatrix}
      A_{12}
    \end{pmatrix}.$$
  We have $\mathrm{rank}(B_{12}) \le \mathrm{rank}(B(Y_1,Y_2))$, for any $Y_1$ and $Y_2$. Hence, $\min_{Y_1,Y_2}\mathrm{rank}(B(Y_1,Y_2))=
    \mathrm{rank}(B(B_{12},B_{12}))=\mathrm{rank}(B_{12})=\mathrm{rank}(A_{12})$. Since $B_{12}$ is symmetric by~(ii), this concludes the proof.
\end{proof}

Let $A(X_2,X_3)=
  \begin{pmatrix}
    A_{11} & A_{12} & A_{13} \\
    A_{21} & X_2    & A_{23} \\
    A_{31} & A_{32} & X_3
  \end{pmatrix}$ be a  block matrix over an arbitrary field in which the blocks $X_2$ and $X_3$ are treated as variables. For matrices over fields of characteristic  different from 2, Cohen et al.~\cite{CJRW89_rank}, see also~\cite{T02_minrank}, proved that

\begin{eqnarray}
  \label{eqn:minRankCohen}
  \min_{X_2,X_3} \mathrm{rank}(A(X_2,X_3)) &= &
  \mathrm{rank}
  \begin{pmatrix} A_{11} & A_{12} & A_{13}
  \end{pmatrix}+
  \mathrm{rank}
  \begin{pmatrix} A_{11} \\  A_{21} \\ A_{31}
  \end{pmatrix}
  +\min\{ \\
  \nonumber
  & &
  \mathrm{rank} \begin{pmatrix}
    A_{11} & A_{12} \\
    A_{31} & A_{32}
  \end{pmatrix}-
  \left(\mathrm{rank}
  \begin{pmatrix}
    A_{11} & A_{12}
  \end{pmatrix}
  +\mathrm{rank}
  \begin{pmatrix}
    A_{11} \\ A_{31}
  \end{pmatrix}\right), \\
  \nonumber
  & &
  \mathrm{rank} \begin{pmatrix}
    A_{11} & A_{13} \\
    A_{21} & A_{23}
  \end{pmatrix}-
  \left(\mathrm{rank}
  \begin{pmatrix}
    A_{11} & A_{13}
  \end{pmatrix}
  +\mathrm{rank}
  \begin{pmatrix}
    A_{11} \\ A_{21}
  \end{pmatrix}\right)\},
\end{eqnarray}

In the following lemma, we prove an upper bound on $\min_{X_2,X_3} \mathrm{rank}(A(X_2,X_3))$, which is equal to the right-hand side of (\ref{eqn:minRankCohen}), if we restrict ourselves to symmetric matrices
$A(X_2,X_3)$. The lemma is valid for the symmetric matrices over an arbitrary field.

\begin{lemma}
  \label{lemma:rankMin2}
  Let  $A_{21}=A_{21}^\top, A_{31}=A_{13}^\top, A_{32}=A_{23}^\top$, and let  $A_{11}$ be symmetric. Then
  \begin{eqnarray}
    \nonumber
    \min_{X_2,X_3} \mathrm{rank}(A(X_2,X_3)) &\le &
    2\mathrm{rank}
    \begin{pmatrix} A_{11} & A_{12} & A_{13}
    \end{pmatrix}+
    \mathrm{rank} \begin{pmatrix}
      A_{11} & A_{12} \\
      A_{31} & A_{32}
    \end{pmatrix}- \\
    \nonumber
    & &
    -\left(\mathrm{rank}
    \begin{pmatrix}
      A_{11} & A_{12}
    \end{pmatrix}
    +\mathrm{rank}
    \begin{pmatrix}
      A_{11} & A_{13}
    \end{pmatrix}\right)
  \end{eqnarray}
  where we minimize over symmetric matrices  $X_2$ and $X_3$.
\end{lemma}

\begin{proof}
  We proceed similarly as in the proof of Lemma~\ref{lemma:rankMin1}, but the argument is slightly more technical.
  Refer to Figure~\ref{fig:minrank21}. We apply symmetric elementary row and column operations to $A(X_2,X_3)$ thereby
  obtaining a 3 by 3 block symmetric matrix $B(Y_2,Y_3)=\begin{pmatrix}
      B_{11} & B_{12} & B_{13} \\
      B_{21} & Y_2    & B_{23} \\
      B_{31} & B_{32} & Y_3
    \end{pmatrix}$, such that
  $B_{ij}$ is obtained from $A_{ij}$, and $Y_2$ and $Y_3$ is obtained from $X_2$ and $X_3$, respectively,  satisfying the following properties.

  \begin{enumerate}[(i)]
    \item Only the block $Y_2$ and $Y_3$ depends on the block $X_2$ and $X_3$, respectively.
    \item  $\mathrm{rank}
            \begin{pmatrix}
              A_{11} & A_{12} & A_{13}
            \end{pmatrix}=
            \mathrm{rank}
            \begin{pmatrix}
              B_{11} & B_{12} & B_{13}
            \end{pmatrix},
            \mathrm{rank}
            \begin{pmatrix}
              A_{11} & A_{12}
            \end{pmatrix}
            =\mathrm{rank}
            \begin{pmatrix}
              B_{11} & B_{12}
            \end{pmatrix}$,
          $\mathrm{rank}
            \begin{pmatrix}
              A_{11} & A_{13}
            \end{pmatrix}
            =\mathrm{rank}
            \begin{pmatrix}
              B_{11} & B_{13}
            \end{pmatrix}$, and \\
          $\mathrm{rank} \begin{pmatrix}
              A_{11} & A_{12} \\
              A_{31} & A_{32}
            \end{pmatrix}=\mathrm{rank} \begin{pmatrix}
              B_{11} & B_{12} \\
              B_{31} & B_{32}
            \end{pmatrix}$.
    \item There exist symmetric matrices $Y_2$ and $Y_3$ such that
          \begin{eqnarray}
            \nonumber
            \mathrm{rank}(B(Y_2,Y_3)) &= &
            2\mathrm{rank}
            \begin{pmatrix} B_{11} & B_{12} & B_{13}
            \end{pmatrix}+
            \mathrm{rank} \begin{pmatrix}
              B_{11} & B_{12} \\
              B_{31} & B_{32}
            \end{pmatrix}- \\
            \nonumber
            & &
            -\left(\mathrm{rank}
            \begin{pmatrix}
              B_{11} & B_{12}
            \end{pmatrix}
            +\mathrm{rank}
            \begin{pmatrix}
              B_{11} & B_{13}
            \end{pmatrix}\right)
          \end{eqnarray}

  \end{enumerate}

  \begin{figure}
    \centering
    $A(X_2,X_3)=\left(\begin{array}{c!{\vline width 1pt}c!{\vline width 1pt}c}
          C      & D      & E   \\  \Xhline{1pt}
          D^\top & X_2    & F   \\  \Xhline{1pt}
          E^\top & F^\top & X_3
        \end{array}\right) \sim
      \left(\begin{array}{c!{\vline width 1pt}c!{\vline width 1pt}c}
          \begin{array}{c|c}
            R_C & 0 \\ \hline
            0   & 0
          \end{array} & \begin{array}{c}
            D_0 \\ \hline
            D_1
          \end{array} & \begin{array}{c}
            E_0 \\ \hline
            E_1
          \end{array} \\ \Xhline{1pt}
          \begin{array}{c|c}
            \ \  \  D_0^\top & D_1^\top
          \end{array} & X_2                        & F                          \\ \Xhline{1pt}
          \begin{array}{c|c}
            \ \ \ E_0^\top & E_1^\top
          \end{array} & F^\top                     & X_3
        \end{array}\right)\sim $ \\

    $
      \left(\begin{array}{c!{\vline width 1pt}c!{\vline width 1pt}c}
          \begin{array}{c|c}
            R_C & 0 \\ \hline
            0   & 0
          \end{array} & \begin{array}{c}
            0   \\ \hline
            D_1
          \end{array} & \begin{array}{c}
            0   \\ \hline
            E_1
          \end{array} \\ \Xhline{1pt}
          \begin{array}{c|c}
            \ \  \ \ \  0 & D_1^\top
          \end{array} & X_2                        & F                          \\ \Xhline{1pt}
          \begin{array}{c|c}
            \ \  \ \  \ 0 & E_1^\top
          \end{array} & F^\top                     & X_3
        \end{array}\right)\sim
      \left(\begin{array}{c!{\vline width 1pt}c!{\vline width 1pt}c}
          \begin{array}{c|c}
            R_C & 0 \\ \hline
            0   & 0
          \end{array} & \begin{array}{c|c}
            0   & 0 \\ \hline
            R_D & 0
          \end{array} & \begin{array}{c|c}
            0   & 0 \\ \hline
            R_E & 0
          \end{array} \\ \Xhline{1pt}
          \begin{array}{c|c}
            \ \  \ \ \  0 & R_D^\top \\ \hline
            \ \  \ \ \  0 & 0
          \end{array} & X_2'                       & F_0                        \\ \Xhline{1pt}
          \begin{array}{c|c}
            \ \  \ \ \  0 & R_E^\top \\ \hline
            \ \  \ \ \  0 & 0
          \end{array} & F_0^\top                   & X_3'
        \end{array}\right)\sim
      \left(\begin{array}{c!{\vline width 1pt}c!{\vline width 1pt}c}
          \begin{array}{c|c}
            R_C & 0 \\ \hline
            0   & 0
          \end{array}  & \begin{array}{c|c}
            0   & 0 \\ \hline
            R_D & 0
          \end{array}  & \begin{array}{c|c}
            0   & 0 \\ \hline
            R_E & 0
          \end{array}  \\ \Xhline{1pt}
          \begin{array}{c|c}
            \ \  \ \ \  0 & R_D^\top \\ \hline
            \ \  \ \ \  0 & 0
          \end{array} & Y_2                         & \begin{array}{c|c}
            \ \ \ \     0     & 0   \\ \hline
            \ \ \ \         0 & F_1
          \end{array} \\ \Xhline{1pt}
          \begin{array}{c|c}
            \ \  \ \ \  0 & R_E^\top \\ \hline
            \ \  \ \ \  0 & 0
          \end{array} & \begin{array}{c|c}
            \ \ \ \  0 & 0        \\ \hline
            \ \ \ \ 0  & F_1^\top
          \end{array} & Y_3
        \end{array}\right)=B(Y_2,Y_3)$
    \caption{Application of the elementary operations to $A(X_2,X_3)$ in order to obtain  $B(Y_2,Y_3)$.}
    \label{fig:minrank21}
  \end{figure}

  Figure~\ref{fig:minrank21} describes  intermediate  steps of the  procedure to obtain $B(Y_2,Y_3)$. For convenience we relabeled in the figure the blocks of $A(X_2,X_3)$ different from $X_2$ and  $X_3$ as\\  $C,D,E,F,C^\top,D^\top,E^\top,F^\top$. In the figure, the blocks of the form $R_X$, where $X\in \{C,D,E,F\}$, are matrices with the full column rank, and $R_C$ has also the full row rank.  By arguments analogous to those used in the proof of Lemma~\ref{lemma:rankMin1}, it is straightforward to check that there exists a sequence of symmetric elementary row and column operations producing a desired matrix at each intermediate step.
  By Lemma~\ref{lemma:rankMin15}, we can choose $Y_2= \begin{pmatrix}
      0 & 0 \\
      0 & G
    \end{pmatrix}$ and $Y_3= \begin{pmatrix}
      0 & 0 \\
      0 & H
    \end{pmatrix}$ as in Figure~\ref{fig:minrank22} such that
  $\mathrm{rank}(F_1)=\mathrm{rank}\begin{pmatrix}
      G        & F_1 \\
      F_1^\top & H   \\
    \end{pmatrix}$, which verifies~(iii) as we will see next.

  On the one hand, by an easy inspection of $B$ we have the following.
  \begin{enumerate}[(a)]
    \item $\mathrm{rank} \begin{pmatrix}
              B_{11} & B_{12} \\
              B_{31} & B_{32}
            \end{pmatrix}=\mathrm{rank}(B_{11})+\mathrm{rank}(B_{12})+\mathrm{rank}(B_{31})+\mathrm{rank}(B_{32})$
    \item  $\mathrm{rank}
            \begin{pmatrix}
              B_{11} & B_{12}
            \end{pmatrix}=\mathrm{rank}(B_{11})+\mathrm{rank}(B_{12})$
    \item   $ \mathrm{rank}
            \begin{pmatrix}
              B_{11} & B_{13}
            \end{pmatrix}=\mathrm{rank}(B_{11})+\mathrm{rank}(B_{13})=\mathrm{rank}(B_{11})+\mathrm{rank}(B_{31})$
    \item      $\mathrm{rank}
            \begin{pmatrix}
              B_{11} & B_{12} & B_{13}
            \end{pmatrix}=\mathrm{rank}(B_{11})+\mathrm{rank} \begin{pmatrix}
              B_{12} & B_{13}
            \end{pmatrix}$
  \end{enumerate}
  On the other hand, $\mathrm{rank}\left(B\left(\begin{pmatrix}
          0 & 0 \\
          0 & G
        \end{pmatrix}, \begin{pmatrix}
          0 & 0 \\
          0 & H\end{pmatrix}\right)\right)$=$$\mathrm{rank}(B_{11})+2\mathrm{rank} \begin{pmatrix}
      B_{12} & B_{13}
    \end{pmatrix}+\mathrm{rank}(B_{32}).$$
  Thus,~(iii) follows by (a)-(d).
  The properties~(i) and~(ii) follow directly from the construction.

  We have that $B(Y_2,Y_3)=CA(X_2,X_3) C^{\top}$, where $C$ is invertible.
  Thus, by~(i) in order to prove the lemma it is enough to prove that
  \begin{eqnarray}
    \nonumber
    \min_{Y_2,Y_3} \mathrm{rank}(B(Y_2,Y_3)) &\le &
    2\mathrm{rank}
    \begin{pmatrix} A_{11} & A_{12} & A_{13}
    \end{pmatrix}+
    \mathrm{rank} \begin{pmatrix}
      A_{11} & A_{12} \\
      A_{31} & A_{32}
    \end{pmatrix}- \\
    \nonumber
    & &
    -\left(\mathrm{rank}
    \begin{pmatrix}
      A_{11} & A_{12}
    \end{pmatrix}
    +\mathrm{rank}
    \begin{pmatrix}
      A_{11} & A_{13}
    \end{pmatrix}\right)
  \end{eqnarray}
  \begin{figure}
    \centering
    $B\left(\begin{pmatrix}
          0 & 0 \\
          0 & G
        \end{pmatrix}, \begin{pmatrix}
          0 & 0 \\
          0 & H\end{pmatrix}\right)=\left(\begin{array}{c!{\vline width 1pt}c!{\vline width 1pt}c}
          \begin{array}{c|c}
            R_C & 0 \\ \hline
            0   & 0
          \end{array} & \begin{array}{c|c}
            0   & 0 \\ \hline
            R_D & 0
          \end{array} & \begin{array}{c|c}
            0   & 0 \\ \hline
            R_E & 0
          \end{array} \\ \Xhline{1pt}
          \begin{array}{c|c}
            \ \  \ \ \  0 & R_D^\top \\ \hline
            \ \  \ \ \  0 & 0
          \end{array} & \begin{array}{c|c}
            \ \ \  0  & 0 \\ \hline
            \ \ \   0 & G
          \end{array} & \begin{array}{c|c}
            \ \ \ \     0     & 0   \\ \hline
            \ \ \ \         0 & F_1
          \end{array} \\ \Xhline{1pt}
          \begin{array}{c|c}
            \ \  \ \ \  0 & R_E^\top \\ \hline
            \ \  \ \ \  0 & 0
          \end{array} & \begin{array}{c|c}
            \ \ \ \  0 & 0        \\ \hline
            \ \ \ \ 0  & F_1^\top
          \end{array} & \begin{array}{c|c}
            \ \ \ \  0 & 0 \\  \hline
            \ \ \ \  0 & H
          \end{array}
        \end{array}\right)$
    \caption{The choice of $Y_2$ and $Y_3$ minimizing the rank of $B(Y_2,Y_3)$.}
    \label{fig:minrank22}
  \end{figure}

  By~(ii) and~(iii), we obtain that the right hand side in the  previous equality
  is indeed an upper bound on $\min_{Y_2,Y_3} \mathrm{rank}(B(Y_2,Y_3))$ which concludes the proof.
\end{proof}


\section{Estimating the $\mathbb{Z}_2$-genus and the Euler $\mathbb{Z}_2$-genus  of $K_{m,n}$}

\label{sec:knm}

\begin{figure}
  \centering
  \includegraphics{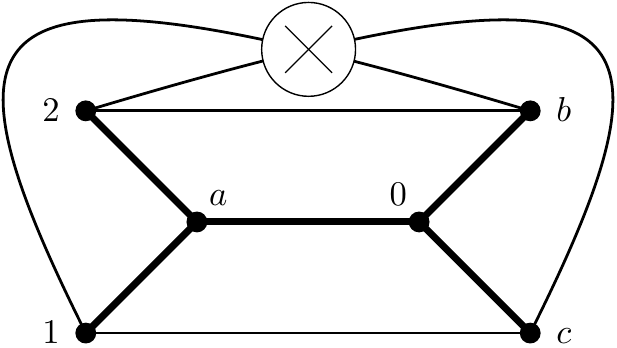}
  \caption{An embedding of $K_{3,3}$ in the plane with a single crosscap. The edges of the spanning tree $T$ are thickened.}
  \label{fig:k33}
\end{figure}

We prove Theorem~\ref{thm:Knm}, whose
proof is based on our previous result~\cite[Lemma 17]{FK18_z2},
which we present next.

In 1976, Kleitman~\cite{Kle76_parity} proved that every drawing of $K_{3,3}$ in the plane contains an odd number of unordered pairs of independent edges crossings an odd number of times. Let $\{a,b,c\}$ and $\{0,1,2\}$ be the two maximal independent sets in $K_{3,3}$
and let $T$ be the spanning tree of $K_{3,3}$ containing all the edges incident to $a$ and $0$. Let $\mathcal{D}$ be a drawing  of $K_{3,3}$ in the plane with finitely many crosscaps such that every pair of independent edges cross an even number of times outside crosscaps and $y_e=0$ for every $e\in E(T)$, see Figure~\ref{fig:k33} for an illustration. The result of Kleitman implies the following lemma, restating~\cite[Lemma 17]{FK18_z2}.

\begin{lemma}
  \label{lemma:Kleitman}
  In the drawing $\mathcal{D}$,
  $y_{b1}^\top y_{c2} + y_{c1}^\top y_{b2} \mod 2 = 1$.
\end{lemma}
\begin{proof}
  Let $\mathcal{D}'$ be the planarization of $\mathcal{D}$.
  By~(\ref{eqn:rel}), $\mathrm{cr}_{\mathcal{D}'}(e,f) \mod 2 = y_e^\top y_f \mod 2$ for every pair of independent edges   $e$ and $f$ in $K_{3,3}$,
  since $\mathcal{D}$ is independently even.
  Using Kleitman's result, $1=\sum_{e,f} \mathrm{cr}_{\mathcal{D}'}(e,f) \mod 2 =\sum_{e,f} y_e^\top y_f \mod 2$, where we sum over unordered independent pairs. Hence, $\sum_{e,f} y_e^\top y_f \mod 2=y_{b1}^\top y_{c2} + y_{c1}^\top y_{b2} \mod 2$ concludes the proof.
\end{proof}

\jk{na prvni pohled mi to pripada jako mix matic a nakresleni... mel jsem dojem, ze jsme to nejdriv redukovali na matici a pak uz pocitali jen s matici.} \jk{Umisteni tehle sekce je trochu divne. Ona celkove se hodi spis k lonskemu z2-genusu nez sem, kdyz uz by se to melo k necemu prilepovat...
  Pro Kmn nemame "partial" matice, ne? Spis by bylo prirozenejsi seskupit dukaz Kmn dohromady, i tu algebraickou cast o turnajovych maticich. Cim dal tim vic mi rozdeleni pripada prirozenejsi.}

\begin{proof}[proof of Theorem~\ref{thm:Knm}]
  We denote the vertices of $K_{m,n}$ in one part by $u_0,\ldots , u_{m-1}$ and in the other part by $v_0,\ldots, v_{n-1}$.

  Let $\mathcal{D}$ be the combinatorial representation of an independently even drawing of $K_{m,n}$ on a surface $S$
  in the plane with finitely many crosscaps (see Section~\ref{sec:combinatorial}). Let  $y_e$ be  the crosscap vector of $e\in E(K_{m,n})$ associated with  $\mathcal{D}$.
  For $i_1,i_2\in \{1,\ldots, m-1\}=[m-1]$, let $A_{i_1i_2}=(a_{j_1j_2})$ be the  $(n-1)\times (n-1)$ matrix with entries
  $a_{j_1j_2}=y_{i_1j_1}^\top y_{i_2j_2}$.
  Let $A=(A_{i_1i_2})$ be the $(m-1)\times (m-1)$ block matrix composed of the previously defined $A_{i_1i_2}$'s.
  By  Lemma~\ref{lemma_forest}, we assume that $y_e=0$ for  $e\in \{u_0v_0, \ldots, u_0v_{n-1}, v_0u_1,\ldots, v_0u_{m-1}\}$, and hence, the matrix $A$ essentially represents~$\mathcal{D}$.
  For every $i_1,i_2,j_1$ and $j_2$, where  $i_1\not= i_2$ and $j_1\not= j_2$ we then apply Lemma~\ref{lemma:Kleitman} to the drawing of $K_{3,3}$ induced by the vertices $u_0,v_0,u_{i_1},u_{i_2},v_{j_1},v_{j_1}$ in $\mathcal{D}$ and
  obtain that  $a_{j_1j_2}+a_{j_2j_1} \mod 2=1$.  In other words, $A_{i_1i_2}$ is a tournament matrix.
  We show that either $A_{i_1i_2}=B+D_{i_1i_2}$ or $A_{i_1i_2}=B+J_{n-1}+D_{i_1i_2}$, where $B$ is a fixed tournament matrix and $D_{i_1i_2}$ is a diagonal matrix.

  If the previous claim holds then  Lemma~\ref{lemma:MatrixRank} applies to $A$. Thus,
  $\mathrm{rank}(A)\ge \left\lceil\frac{(m-2)(n-2)}{2}\right\rceil-(m-3)$.
  By Corollary~\ref{cor:rank}, $\mathrm{eg}_0(K_{m,n}) \ge \left\lceil\frac{(m-2)(n-2)}{2}\right\rceil-(m-3)$ as desired. Similarly, $2\mathrm{g}_0(K_{m,n})\ge\left\lceil\frac{(m-2)(n-2)}{2}\right\rceil-(m-3)$ and the claimed lower bound for $\mathrm{g}_0(K_{m,n})$ follows as well.

  It remains to prove the claim. To this end we apply the argument that was used to prove~\cite[Lemma 17]{FK18_z2}.
  In the drawing of $K_{3,3}$ induced by the vertices $u_0,v_0,u_{i_1},u_{i_2},v_{j_1},v_{j_1}$ in $\mathcal{D}$
  we locally deform  $\mathcal{D}$ in a close neighborhood of $u_0$, so that the edges  $u_0v_0,u_0v_{j_1}$ and $u_0v_{j_2}$ cross one another an even number of times, while keeping $\mathcal{D}$ independently even. It is easy to see that this is indeed possible.
  Similarly, we adjust the drawing in a close neighborhood of $v_0$, so that the edges  $v_0u_0,v_0u_{i_1}$ and $v_0u_{i_2}$ cross one another an even number of times.  Let $\mathcal{D}'$ be the resulting modification of $\mathcal{D}$.

  We will prove below that  \\

  (*) In the block $A_{i_1,i_2}$, for $j_1\not=j_2$, the value
  $a_{j_1,j_2} =1$ if and only if up to the choice of orientation the edges  $u_0v_0,u_0v_{j_1},u_0v_{j_2}$  and
  $v_0u_0,v_0u_{i_1},v_0u_{i_2}$ appear in the rotation at $u_0$ and $v_0$,  respectively, in this order clockwise. \\

  Hence, suppose that (*) holds and that $v_0u_0,v_0u_{i_1},v_0u_{i_2}$  appear in the rotation at $v_0$  in $\mathcal{D}'$ in this order clockwise.
  For $i_1',i_2'\in [m-1]$, $i_1'\not=i_2'$, we adjust the drawing in a close neighborhood of $v_0$, so that the edges  $v_0u_0,v_0u_{i_1'}$ and $v_0u_{i_2'}$ cross one another an even number of times.
  Let $\mathcal{D}''$ be the resulting drawing.
  By~(*), $A_{i_1'i_2'}=A_{i_1i_2}+D_{i_1'i_2'}$, where $D_{i_1i_2}$ is a diagonal matrix,
  if  $v_0u_0,v_0u_{i_1'},v_0u_{i_2'}$ in $\mathcal{D}''$ appear in the rotation at $v_0$  in this order clockwise;
  and $A_{i_1'i_2'}=A_{i_1i_2}+D_{i_1'i_2'}+J_{n-1}$,
  if
  $v_0u_0,v_0u_{i_1'},v_0u_{i_2'}$ appear in the rotation at $v_0$ in $\mathcal{D}''$ in this order counterclockwise. It remains to prove~(*).

  Let $\gamma_{i,j}$  be the closed curve representing the cycle traversing vertices $u_0,v_0,u_i$ and $v_j$ in $\mathcal{D}$.
  The condition that characterizes when $a_{j_1j_2} =1$, for $j_1\not=j_2$, follows by considering a slightly perturbed drawing of $\gamma_{i_1,j_1}$ and $\gamma_{i_2,j_2}$, in which all their intersections become  proper edge crossings. Note that  $a_{j_1j_2} =1$ if and only if $u_{i_1}v_{j_i}$ and $u_{i_2}v_{j_2}$
  have an odd number of intersections at crosscaps. Furthermore,  $\gamma_{i_1,j_1}$ and $\gamma_{i_2,j_2}$ must have an even number of intersections in total. Therefore as $\mathcal{D}$ is an independently even drawing,  $a_{j_1j_2} =1$ if and only if in $\mathcal{D}'$  the edge $u_0v_0$ is a transversal intersection of $\gamma_{i_1,j_1}$ and $\gamma_{i_2,j_2}$; in other words, up to the choice of orientation $u_0v_0,u_0v_{j_1},u_0v_{j_2}$  and
  $v_0u_0,v_0u_{i_1},v_0u_{i_2}$ appear in the rotation at $u_0$  in this order clockwise.
\end{proof}


\section{Amalgamations}

\label{sec:amalgamations}

\subsection{1-amalgamations}
\label{sec:1-amalgamations}

In order to ease up the readability, as a warm-up we first  reprove the result of Schaefer and \v{S}tefankovi\v{c} for the Euler genus. The proof of our result for 2-amalgamations follows the same blueprint, but the argument gets slightly more technical.

\begin{theorem}[\cite{SS13_block}]

  \label{thm:1amalg}
  Let $G_1$ and $G_2$ be graphs.
  Let $G=\amalg_{v}(G_1,G_2)$. Then
  $\mathrm{eg}_0(G_1)+\mathrm{eg}_0(G_2)= \mathrm{eg}_0(G)$.
\end{theorem}
\begin{proof}
  The Euler $\mathbb{Z}_2$-genus of a graph is the sum of the Euler $\mathbb{Z}_2$-genera of its connected components~\cite[Lemma 7]{SS13_block}. Thus, we assume that both $G_1$ and $G_2$ are connected.

  We start the argument similarly as in~\cite{SS13_block} by choosing an appropriate spanning tree $T$ in $G$ and fixing an independently even drawing of $G$ on $N_g$, in which each edge  in $E(T)$ passes an even number of times through each  crosscap. Nevertheless, the rest of the proof differs considerably, one of the key differences being the  use of Corollary~\ref{cor:rank} rather than Lemma~\ref{lemma_smallgenus2}
  to bound the Euler $\mathbb{Z}_2$-genus of involved graphs.

  The following claim is rather easy to prove.
  \begin{claim}
    \label{claim:easy1amalgamation}
    \begin{equation}
      \label{eqn:easy1amalgamation}
      \mathrm{eg}_0(G)\le \mathrm{eg}_0(G_1) + \mathrm{eg}_0(G_2)
    \end{equation}
  \end{claim}

  \begin{proof}
    Given an independently even drawing $\mathcal{D}_1$ and $\mathcal{D}_2$ of $G_1$ and $G_2$ on $N_{g_1}$ and $N_{g_2}$, respectively, we construct an independently even drawing of $G$ on $N_{g_1+g_2}$ in two steps.
    First, we construct an independently even drawing  $\mathcal{D}$ of the disjoint union of $G_1$ and $G_2$ by combining
    $\mathcal{D}_1$ and $\mathcal{D}_2$ on $N_{g_1+g_2}$, which is obtained as a connected sum of  $N_{g_1}$ and $N_{g_2}$. Without loss of generality we  assume that the both copies of $v$
    are incident to the same connected component of the complement of $\mathcal{D}$ in $N_{g_1+g_2}$.
    Thus, they can be joined by a crossing free edge $e$. Contracting $e$ then results in an independently even drawing of $G$ on $N_{g_1+g_2}$.
  \end{proof}

  It remains to prove  the opposite inequality. We first choose a spanning tree $T$ of $G$ with the following property. Recall that $v$ is a fixed cut vertex.
  For every $e\in E(G)\setminus E(T)$ it holds that if $v\not\in e$ then the unique cycle in $T\cup e$ does not pass through $v$. The desired spanning tree $T$ is obtained as the exploration tree of a Depth-First-Search in $G$ starting at $v$.  We consider  an independently even drawing of $G$ on a surface $S$ witnessing its Euler genus.
  By Lemma~\ref{lemma_forest}, we obtain a drawing $\mathcal{D}$ of $G$ in the plane with finitely many crosscaps in which every edge of $T$ passes through each crosscap an even number of times,
  and independent pairs of edges cross an even number of times outside the crosscaps.

  First, a few words on the strategy of the rest of the proof.
  Let   $B$ be a matrix essentially representing the planarization of $\mathcal{D}$, whose entries represent the parity of crosscap crossings in $\mathcal{D}$ between the edges of $G$ not belonging to its spanning tree $T$.
  For a pair of edges $e$ and $f$ in $G$ this parity is given by $y_e^{\top}y_f \mod 2$. By introducing an appropriate block structure on $B$, and using Lemma~\ref{lemma:rankMin1} we show that the rank of $B$ can be lower bounded by $\mathrm{eg}_0(G_1)+\mathrm{eg}_0(G_2)$.
  This will conclude the proof since the rank of $B$ is easily upper bounded by $\mathrm{eg}_0(G)$.

  Let $E_1$ and $E_2$ be the set of edges in  $E(G_1)\setminus E(T)$
  and  $E(G_2)\setminus E(T)$, respectively, that are not incident to $v$.
  Let $F_1$ and $F_2$ be the set of edges in  $E(G_1)\setminus E(T)$
  and  $E(G_2)\setminus E(T)$, respectively, that are  incident to $v$.

  Let $\alpha,\beta \in \{E_1,E_2,F_1,F_2\}$.
  Let $\alpha = \{e_1,\ldots, e_{|\alpha|}\}$.
  Let $\beta = \{e_1',\ldots, e_{|\beta|}'\}$.
  Let $A_{\alpha,\beta}=(a_{ij})$ be the $|\alpha|\times |\beta|$ matrix
  over $\mathbb{Z}_2$ such that $a_{ij}=y_{e_i}^{\top}y_{e_j'}$.
  Let $B=(B_{ij})$ be a $4\times 4$ block matrix such that $B_{ij}=A_{\alpha_i,\alpha_j}$, where $\alpha_1=E_1, \alpha_2=F_1, \alpha_3 =F_2$ and $\alpha_4=E_2$. Clearly,  $B$  essentially represents the planarization of $\mathcal{D}$.

  In what follows we collect some properties of $B$ and its submatrices, whose combination establishes the result.
  The rank of $B$ is at most the dimension of the space generated by the crosscap vectors. The latter is at most  $\mathrm{eg}_0(G)$ since crosscap vectors have $\mathrm{eg}_0(G)$ or $\mathrm{eg}_0(G)+1$ coordinates depending on whether the original drawing  of $G$ is on $N_g$ or $M_g$, but in the latter we loose one dimension since every crosscap vector has an even number of ones.
  Hence, we have
  \begin{equation}
    \label{eqn:egVsRank1}
    \mathrm{eg}_0(G)\ge \mathrm{rank}(B).
  \end{equation}
  If we  arbitrarily change blocks $A_{F_1,F_1}$ and $A_{F_2,F_2}$, $B$ will still essentially represent the planarization of $\mathcal{D}$.
  Let $B_1(X)=
    \begin{pmatrix}
      A_{E_1,E_1} & A_{E_1,F_1} \\
      A_{F_1,E_1} & X
    \end{pmatrix}$ and
  $B_2(X) =
    \begin{pmatrix}
      X           & A_{F_2,E_2} \\
      A_{E_2,F_2} & A_{E_2,E_2}
    \end{pmatrix}$.
  Then by  Corollary~\ref{cor:rank},
  \begin{equation}
    \label{eqn:blockBound1}
    \mathrm{eg}_0(G_i)\le \min_{X} \mathrm{rank}(B_i(X)),
  \end{equation}
  where we minimize over symmetric matrices $X$.
  By Lemma~\ref{lemma:rankMin1},
  \begin{equation}
    \label{eqn:Lemma61}
    \min_{X} \mathrm{rank}(B_i(X))= 2\mathrm{rank}\begin{pmatrix}
      A_{E_i,E_i} & A_{E_i,F_i} \end{pmatrix}-\mathrm{rank}(A_{E_i,E_i}).
  \end{equation}
  The last ingredient in the proof is the following claim which holds due to  the careful choice of the spanning tree $T$.
  \begin{claim}
    \label{claim:lastIngredient1}
    $2\left(\mathrm{rank}\begin{pmatrix}
        A_{E_1,E_1} & A_{E_1,F_1} \end{pmatrix}+\mathrm{rank}\begin{pmatrix}
        A_{E_2,E_2} & A_{E_2,F_2} \end{pmatrix}\right)-\\ -(\mathrm{rank}(A_{E_1,E_1})+\mathrm{rank}(A_{E_2,E_2}))\le  \mathrm{rank}(B)$
  \end{claim}

  \begin{proof}
    We start by giving an outline of the proof.
    First, we modify $B$ appropriately by elementary row operations and refine its block structure
    thereby obtaining a   6 by 4 block matrix $C=(C_{ij})$, which is not
    necessarily symmetric, see Figure~\ref{fig:LemmaRank11}.
    Second, we  permute the rows and columns of blocks in $C$ so that all the above diagonal blocks are zero matrices with respect to a certain block diagonal. The claimed lower bound on  $\mathrm{rank}(B)$ is then obtained by summing up ranks of the diagonal blocks of $C$. We proceed with the detailed description of the proof.

    \begin{figure}
      \centering
      $B=\left(
        \begin{array}{c|c}
            \begin{array}{c|c}
              B_{11} & B_{12} \\ \hline
              B_{21} & B_{22}
            \end{array} &
            \begin{array}{cc}
              0 \ \ \       & 0 \\
              B_{23} \ \ \  & 0
            \end{array} \\
            \hline
            \begin{array}{cc}
              0 & \ \ \ B_{32} \\
              0 & 0
            \end{array} &
            \begin{array}{c|c}
              B_{33} & B_{34} \\ \hline
              B_{43} & B_{44}
            \end{array}
          \end{array}\right)\sim
        \left(
        \begin{array}{c|c}
            \begin{array}{c|c}
              0      & C_{12} \\
              C_{21} & C_{22} \\ \hline
              C_{31} & C_{32}
            \end{array} &
            \begin{array}{cc}
              0 \ \ \       & 0 \\
              0 \ \ \       & 0 \\
              C_{33} \ \ \  & 0
            \end{array} \\
            \hline
            \begin{array}{cc}
              0 & \ \ \ C_{42} \\
              0 & 0            \\
              0 & 0
            \end{array} &
            \begin{array}{c|c}
              C_{43} & C_{44} \\ \hline
              C_{53} & C_{54} \\
              C_{63} & 0
            \end{array}
          \end{array}\right)=C$
      \caption{Application of the elementary row operations to $B$ in order to obtain  $C$.}
      \label{fig:LemmaRank11}
    \end{figure}

    Since $\mathcal{D}$ is an independently even drawing, the following blocks of $B$ are zero matrices: $B_{14}=A_{E_1.E_2},B_{41}=A_{E_2,E_1},B_{13}=A_{E_1,F_2},B_{31}=A_{F_2,E_1},B_{24}=A_{E_2,F_1},B_{42}=A_{F_1,E_2}$.
    Indeed, by the choice of $T$ if $e\in E_1$ and $f\in E_2\cup F_2$, or $e\in E_2$ and $f\in E_1\cup F_1$,
    then $T\cup \{e,f\}$ contains a pair of vertex disjoint cycles $C_e$ and $C_f$
    containing $e$ and $f$, respectively.
    Thus, by~\cite[Lemma 1]{SS13_block}\footnote{\label{foot:note}The lemma is a simple consequence of the fact, that a pair of closed curves in the plane intersecting in a finite number of proper crossings, cross an even number of times.} we have $y_{e}^{\top}y_{f}=0$.

    In order to construct $C$, we first obtain $\begin{pmatrix}B_{11}' & B_{12}'\end{pmatrix}$ from $\begin{pmatrix}
        B_{11} & B_{12} \end{pmatrix}$ by elementary row operations, where
    $B_{11}'$ and $B_{12}'$ is obtained from $B_{11}$ and $B_{12}$, respectively, so that
    $B_{11}'=\begin{pmatrix}
        0 \\ C_{21}
      \end{pmatrix}$, where  $C_{21}$ has full row rank.
    Note that by performing the row operations on $\begin{pmatrix}B_{11} & B_{12} & B_{13} & B_{14}\end{pmatrix}$ rather than just
    $\begin{pmatrix}B_{11} & B_{12}\end{pmatrix}$, we do not affect the  blocks $B_{13}$ and $B_{14}$, since they are zero matrices.
    Let $B_{13}=B_{13}'$ and $B_{14}=B_{14}'$.
    The remaining blocks   in the first row are then refined accordingly into $\begin{pmatrix}C_{1j} \\ C_{2j}\end{pmatrix}=B_{1j}'$.
    Note that
    \begin{eqnarray}
      \nonumber
      2\mathrm{rank}\begin{pmatrix} A_{E_1,E_1} &  A_{E_1,F_1} \end{pmatrix}-\mathrm{rank}(A_{E_1,E_1})&=& \mathrm{rank}\begin{pmatrix}
        B_{11} \\ B_{21} \end{pmatrix}+(\mathrm{rank}\begin{pmatrix}B_{11} & B_{12}\end{pmatrix}-\mathrm{rank}(B_{11})) \\
      &=& \mathrm{rank}\begin{pmatrix} C_{21} \\C_{31}\end{pmatrix}+\mathrm{rank}(C_{12})
      \label{eqn:blockRank1}
    \end{eqnarray}

    Similarly, we obtain $\begin{pmatrix}B_{43}' & B_{44}'\end{pmatrix}$ from $\begin{pmatrix} B_{43} & B_{44}\end{pmatrix}$ by elementary row operations, where
    $B_{43}'$ and $B_{44}'$ is obtained from $B_{43}$ and $B_{44}$, respectively, so that
    $B_{44}'=\begin{pmatrix}C_{54} \\ 0 \end{pmatrix}$, where  $C_{54}$ has full row rank.
    Let $B_{41}=B_{41}'$ and $B_{42}=B_{42}'$.
    The remaining blocks in the first row are then refined accordingly into $\begin{pmatrix}C_{5j} \\ C_{6j}\end{pmatrix}=B_{4j}'$.
    Note that
    \begin{eqnarray}
      \nonumber
      2\mathrm{rank}\begin{pmatrix}
        A_{E_2,E_2} & A_{E_2,F_2}
      \end{pmatrix}-\mathrm{rank}(A_{E_2,E_2})&=& \mathrm{rank} \begin{pmatrix}B_{44} \\ B_{34}\end{pmatrix}+(\mathrm{rank}\begin{pmatrix} B_{44} & B_{43} \end{pmatrix}- \\ & & -\mathrm{rank}(B_{44})) \nonumber \\
      &=& \mathrm{rank}\begin{pmatrix} C_{54} \\ C_{44} \end{pmatrix}+\mathrm{rank}(C_{63})
      \label{eqn:blockRank2}
    \end{eqnarray}

    We neither refine nor change in any way the second and third row of blocks and therefore we have $C_{3j}=B_{2j}$ and $ C_{4j}=B_{3j}$ for $j=1,2,3,4$.

    \begin{figure}
      \centering
      $\begin{pmatrix}
          0      & C_{12} & 0      & 0      \\
          C_{21} & C_{22} & 0      & 0      \\
          C_{31} & C_{32} & C_{33} & 0      \\
          0      & C_{42} & C_{43} & C_{44} \\
          0      & 0      & C_{53} & C_{54} \\
          0      & 0      & C_{63} & 0      \\
        \end{pmatrix}\sim
        \begin{pmatrix}
          C_{12} & 0      & 0      & 0      \\
          0      & C_{63} & 0      & 0      \\
          C_{22} & 0      & C_{21} & 0      \\
          C_{32} & C_{33} & C_{31} & 0      \\
          0      & C_{53} & 0      & C_{54} \\
          C_{42} & C_{43} & 0      & C_{44}
        \end{pmatrix}$
      \caption{Permuting the rows and columns in $C$.}
      \label{fig:LemmaRank12}
    \end{figure}

    Finally, we permute the rows and columns of blocks in $C$ by applying  the permutation $\pi_r:[6]\rightarrow [6]$ and $\pi_c: [4] \rightarrow [4]$, respectively, see Figure~\ref{fig:LemmaRank12}.
    We specify permutations as vectors $(\pi_r(1),\ldots,\pi_r(6))=(1,6,2,3,5,4)$ and $(\pi_c(1),\ldots,\pi_c(4))=(2,3,1,4)$.
    It is a routine to check that $C_{12},C_{63}, \begin{pmatrix} C_{21} \\ C_{31} \end{pmatrix}$, and  $\begin{pmatrix} C_{54} \\ C_{44} \end{pmatrix}$ form a diagonal
    in this order after permuting the rows and columns of blocks and that all the above diagonal blocks are zero matrices. Hence, summing up the equalities~(\ref{eqn:blockRank1}) and~(\ref{eqn:blockRank2}), and observing that the sum of the ranks of the diagonal blocks is a lower bound on $\mathrm{rank}(B)$ concludes the proof.
  \end{proof}

  We are done by the following chain of (in)equalities. 
  \begin{align*}
    \mathrm{eg}_0(G)  & \mathrel{\substack{(\ref{eqn:easy1amalgamation})\\\le\\\\ }}  \mathrm{eg}_0(G_1) + \mathrm{eg}_0(G_2) \\ & \mathrel{\substack{(\ref{eqn:blockBound1})\\\le\\\\ }}  \min_{X} \mathrm{rank}(B_1(X))+  \min_{X} \mathrm{rank}(B_2(X)) \\
     & \mathrel{\substack{(\ref{eqn:Lemma61})\\ =\\\\ }}   2\mathrm{rank}\begin{pmatrix}
      A_{E_1,E_1} & A_{E_1,F_1} \end{pmatrix}+2\mathrm{rank}\begin{pmatrix}
      A_{E_2,E_2} & A_{E_2,F_2} \end{pmatrix}- 
      \mathrm{rank}(A_{E_1,E_1})-\mathrm{rank}(A_{E_2,E_2}) \\
      & \le  \mathrm{rank}(B) \\  & \mathrel{\substack{(\ref{eqn:egVsRank1})\\ \le \\\\ }}  \mathrm{eg}_0(G)
  \end{align*}

\end{proof}

\subsection{2-amalgamations}
\label{sec:2-amalgamations}

\begin{proof}[proof of Theorem~\ref{thm:2amalg}]
  We will prove the parts \rm{a)} and  \rm{b)} in parallel.
  We assume that $G-u-v$ has precisely  2 connected components and that $uv\not\in E(G)$ (the general case is treated in the appendix). By the block additivity result~\cite{SS13_block}, we assume that none of $u$ and $v$ is a cut vertex of $G$, and by the additivity over connected components~\cite[Lemma 7]{SS13_block} that $G$ is connected.
  We follow the line of thought analogous to the proof  of Theorem~\ref{thm:1amalg}.

  It is easy to prove the following claim.
  \begin{claim}
    \label{claim:easy1amalgamation2}
    \begin{equation}
      \label{eqn:easy1amalgamation2}
      \mathrm{g}_0(G)\le \mathrm{g}_0(G_1) + \mathrm{g}_0(G_2)+1 \ \ \mathrm{and} \ \    \mathrm{eg}_0(G)\le \mathrm{eg}_0(G_1) + \mathrm{eg}_0(G_2)+2
    \end{equation}
  \end{claim}

  \begin{proof}
    Given an independently even drawing $\mathcal{D}_1$ and $\mathcal{D}_2$ of $G_1$ and $G_2$ on $S_1$ and $S_2$, respectively, we construct an independently even drawing of $G$ on
    a surface obtained from a connected sum of $S_1$ and $S_2$ by attaching to it a single handle.
    Similarly as in the proof of~(\ref{eqn:easy1amalgamation}), we construct an independently even drawing  $\mathcal{D}$ of the disjoint union of $G_1$ and $G_2$ on the connected sum $S$ of $S_1$ and $S_2$ in which we identify the two copies of $u$. The two copies of $v$ are then identified after attaching a handle on $S$ joining small neighborhoods of the two copies of $v$.
  \end{proof}

  It remains to prove the opposite inequalities of \rm{a)} and \rm{b)}.
  We choose an appropriate spanning tree $T$ of $G$ and fix an independently even drawing of $G$ on $N_g$, in which each edge of $T$ passes an even number of times through each crosscap. To this end we first choose a spanning tree $T'$ of $G-v$ with the following property.
  For every $e\in E(G-v)\setminus E(T')$, if $u\not\in e$ then the unique cycle in $T'\cup e$ does not pass through $u$. The desired spanning $T'$ is obtained as the exploration tree of a Depth-First-Search in $G-v$ starting at $u$.
  Let $u_i$ be an arbitrary vertex such that $vu_i\in E(G_i)$, for $i=1,2$.
  We obtain $T$ as $T'\cup vu_1$.

  We consider  an independently even drawing of $G$ on a surface $S$ witnessing its genus (respectively, Euler genus).
  By Lemma~\ref{lemma_forest}, we obtain a drawing $\mathcal{D}$ in the plane with finitely many crosscaps in which every edge of $T$ passes through each crosscap an even number of times,
  and independent pairs of edges cross an even number of times outside the crosscaps.
  Let $y_e$ be the crosscap vector of $e\in E(G)$ associated with  $\mathcal{D}$.

  First, a few words on the strategy of the rest of the proof.
  Let $B$ be a matrix essentially representing the planarization of $\mathcal{D}$, whose entries represent the parity of crosscap crossings in $\mathcal{D}$ between the edges of $G$ not belonging to its spanning tree $T$.
  For a pair of edges $e$ and $f$ in $G$ this parity is given by $y_e^{\top}y_f \mod 2$.
  By introducing an appropriate block structure on $B$, and using Lemma~\ref{lemma:rankMin2} we show that the rank of $B$ can be lower bounded by $\mathrm{g}_0(G_1)+\mathrm{g}_0(G_2)-7/2$ (respectively,  $\mathrm{eg}_0(G_1)+\mathrm{eg}_0(G_2)-3$).
  This will conclude the proof in this case, since the rank of $B$ is easily upper bounded by $2\mathrm{g}_0(G)$ (respectively, $\mathrm{eg}_0(G)$).

  Let $E_1$ and $E_2$ be the set of edges in  $E(G_1)\setminus E(T)$
  and  $E(G_2)\setminus E(T)$, respectively, that are incident neither to $v$ nor to $u$.
  Let $F_1$ and $F_2$ be the set of edges in  $E(G_1)\setminus E(T)$
  and  $E(G_2)\setminus E(T)$, respectively, that are  incident to $u$.
  Let $H_1$ and $H_2$ be the set of edges in  $E(G_1)\setminus E(T)$
  and  $E(G_2)\setminus E(T)$, respectively, that are  incident to $v$.
  Thus, we have that $E(T),E_1,E_2,F_1,F_2,H_1$ and $H_2$ form a partition of $E(G)$.

  Let $\alpha,\beta \in \{E_1,E_2,F_1,F_2,H_1,H_2\}$.
  Let $\alpha = \{e_1,\ldots, e_{|\alpha|}\}$.
  Let $\beta = \{e_1',\ldots, e_{|\beta|}'\}$.
  Let $A_{\alpha,\beta}=(a_{ij})$ be the $|\alpha|\times |\beta|$ matrix
  over $\mathbb{Z}_2$ such that $a_{ij}=y_{e_i}^{\top}y_{e_j'} \mod 2$.
  Let $B=(B_{ij})$ be a 6 by 6 block matrix such that $B_{ij}=A_{\alpha_i,\alpha_j}$, where $\alpha_1=E_1, \alpha_2=F_1, \alpha_3=H_1, \alpha_4=H_2, \alpha_5 =F_2$ and $\alpha_6=E_2$. Clearly,  $B$  essentially represents the planarization of $\mathcal{D}$.

  In what follows we collect some properties of $B$ and its submatrices, whose combination establishes the result.
  Since the rank of $B$ is at most the dimension of the space generated by the crosscap vectors, we have the following
  \begin{equation}
    \label{eqn:egVsRank21}
    2\mathrm{g}_0(G)\ge \mathrm{rank}(B) \ \ (\mathrm{respectively}, \ \ \mathrm{eg}_0(G)\ge \mathrm{rank}(B)).
  \end{equation}
  Let $B_1(X_2,X_3)=
    \begin{pmatrix}
      A_{E_1,E_1} & A_{E_1,F_1} & A_{E_1,H_1} \\
      A_{F_1,E_1} & X_2         & A_{F_1,H_1} \\
      A_{H_1,E_1} & A_{H_1,F_1} & X_3         \\
    \end{pmatrix}$,
  $B_2(X_1,X_2) =
    \begin{pmatrix}
      X_1         & A_{H_2,F_2} & A_{H_2,E_2} \\
      A_{F_2,H_2} & X_2         & A_{F_2,E_2} \\
      A_{E_2,H_2} & A_{E_2,F_2} & A_{E_2,E_2}
    \end{pmatrix}$. \\

  Since changing blocks $A_{F_i,F_i}$ and $A_{H_i,H_i}$, for $i=1,2$, except for the diagonal,  does not affect the property that $B$ essentially represents the planarization of $\mathcal{D}$,
  by  Corollary~\ref{cor:rank},
  \begin{equation}
    \label{eqn:blockBound2}
    2\mathrm{g}_0(G_i)\le \min_{X_2,X_3} \mathrm{rank}(B_i(X_2,X_3))+2 \ \  (\mathrm{respectively}, \ \ \mathrm{eg}_0(G_i)\le \min_{X_2,X_3} \mathrm{rank}(B_i(X_2,X_3))),
  \end{equation}
  where we minimize over symmetric matrices.
  We add 2 on the right hand side in the first inequality  due to the orientability. In particular, it can happen that  $X_2$ or $X_3$ minimizing the rank has a $1$-entry on the diagonal. If this is the case,  in the corresponding independently even drawing, as constructed in the proof of Proposition~\ref{prop_strenghtening}, there exists an edge $e$ incident to $u$ or $v$ such that $y_e^\top y_e \mod 2=1$. In order to make  $y_e^\top y_e \mod 2=0$, we introduce  a crosscap, and push the edge $e$ over it. The introduced crosscap can be shared by the edges incident to $v$ and by the edges incident to $u$. Therefore adding 2  crosscaps is sufficient.
  By Lemma~\ref{lemma:rankMin2},
  \begin{eqnarray}
    \nonumber
    \min_{X_2,X_3} \mathrm{rank}(B_i(X_2,X_3)) &\le &
    2\mathrm{rank}
    \begin{pmatrix} A_{E_i,E_i} & A_{E_i,F_i} & A_{E_i,H_i}
    \end{pmatrix}+
    \mathrm{rank} \begin{pmatrix}
      A_{E_i,E_i} & A_{E_i,F_i} \\
      A_{H_i,E_i} & A_{H_i,F_i}
    \end{pmatrix}- \\
    \label{eqn:Lemma62}
    & &
    -\left(\mathrm{rank}
    \begin{pmatrix}
      A_{E_i,E_i} & A_{E_i,F_i}
    \end{pmatrix}
    +\mathrm{rank}
    \begin{pmatrix}
      A_{E_i,E_i} & A_{E_i,H_i}
    \end{pmatrix}\right)
  \end{eqnarray}

  The last ingredient in the proof is the following claim which holds due to  the careful choice of the spanning tree $T$.
  \begin{claim}
    \label{claim:lastIngredient2}
    $\min_{X_2,X_3} \mathrm{rank}(B_1(X_2,X_3)) + \min_{X_1,X_2} \mathrm{rank}(B_2(X_1,X_2))\le  \mathrm{rank}(B)+3$,
    where we minimize over symmetric matrices.
  \end{claim}
  \begin{proof}
    We start by giving an outline of the proof.
    First, we modify $B$ appropriately by elementary row operations and refine its block structure
    thereby obtaining a   12 by 6 block matrix $C=(C_{ij})$, which is not
    necessarily symmetric.
    Second, we  permute the rows and columns of blocks in $C$ so that all the above diagonal blocks are zero matrices with respect to a certain block diagonal. The claimed lower bound on  $\mathrm{rank}(B)$ is then obtained by using~(\ref{eqn:Lemma62}) and summing up ranks of the diagonal blocks of $C$. We proceed with the detailed description of the proof.

    \begin{figure}
      \centering
      $B=\left(
        \begin{array}{c|c}
            \begin{array}{c|c|c}
              B_{11} & B_{12} & B_{13} \\ \hline
              B_{21} & B_{22} & B_{23} \\ \hline
              B_{31} & B_{32} & B_{33}
            \end{array} &
            \begin{array}{cll}
              0           & 0 \           & 0 \\
              0           & B_{25} \ \ \  & 0 \\
              B_{34} \ \  & 0 \           & 0
            \end{array} \\
            \hline
            \begin{array}{crc}
              0 & 0         & \   B_{43} \\
              0 & \  B_{52} & 0          \\
              0 & 0         & 0
            \end{array} &
            \begin{array}{c|c|c}
              B_{44} & B_{45} & B_{46} \\ \hline
              B_{54} & B_{55} & B_{56} \\ \hline
              B_{64} & B_{65} & B_{66}
            \end{array}
          \end{array}\right)\sim$

      \medskip \medskip
      $\sim\left(
        \begin{array}{c|c}
            \begin{array}{c|c|c}
              \begin{array}{c} 0 \\ 0 \\ C_{31} \end{array} & \begin{array}{c} 0 \\  C_{22} \\ C_{32} \end{array} & \begin{array}{c} C_{13} \\  C_{23} \\ C_{33} \end{array} \\ \hline
              C_{41}                      & C_{42}                      & C_{43}                      \\ \hline
              \begin{array}{c} C_{51} \\ 0 \end{array} & \begin{array}{c} C_{52} \\  C_{62} \end{array} & \begin{array}{c} C_{53} \\ C_{63}
              \end{array}
            \end{array} &
            \begin{array}{lll}
              0 \ \ \ \  \ \            & 0 \ \ \            & 0 \\
              0 \ \ \ \  \ \            & 0 \  \ \           & 0 \\
              0 \  \ \ \  \ \           & 0 \ \ \            & 0 \\
              0 \ \  \ \  \ \           & C_{45} \ \ \  \ \  & 0 \\
              C_{54} \ \  \  \ \ \ \ \  & 0 \ \ \            & 0 \\
              C_{64}  \ \ \   \ \ \ \   & 0 \ \ \            & 0
            \end{array} \\
            \hline
            \begin{array}{crr}
              0 & \ \ 0        & \ \ \  \  C_{73} \\
              0 & \ \ 0        & \ \ \  \  C_{83} \\
              0 & \ \ \ C_{92} & \  \  \ 0        \\
              0 & \ \ 0        & \  \ \ 0         \\
              0 & \ \ 0        & \ \  \  0        \\
              0 & \ \ 0        & \ \ \ 0          \\
            \end{array} &
            \begin{array}{c|c|c}
              \begin{array}{c} C_{74} \\ C_{84} \end{array} & \begin{array}{c} C_{75} \\ C_{85} \end{array} & \begin{array}{c} 0 \\ C_{86}
              \end{array} \\ \hline
              C_{94}                      & C_{95}                      & C_{96}                      \\ \hline
              \begin{array}{c} C_{10,4} \\ C_{11,4} \\ C_{12,4} \end{array} & \begin{array}{c} C_{10,5} \\  C_{11,5} \\ 0 \end{array} & \begin{array}{c} C_{12,6} \\  0 \\ 0 \end{array}
            \end{array}
          \end{array}\right)=C$
      \caption{Application of the elementary row operations to $B$ in order to obtain  $C$.}
      \label{fig:LemmaRank21}
    \end{figure}

    Note that the edges of $E(G)\setminus E(T)$, on the one side, and rows and columns, respectively, of $B$, on the other side, are in a one-to-one correspondence, where the entry in the row corresponding to $e$ and the column corresponding to $f$ equals to  $y_e^\top y_f \mod 2$.

    First, for every $e\in H_2$, $e\not=vu_2$, we add the row and column of $B$ corresponding to $y_{vu_2}$ to the row and column, respectively, corresponding to $y_{e}$ and delete the row and column corresponding to $y_{vu_2}$. Slightly abusing notation, let $B$ be the resulting matrix.
    Since $\mathcal{D}$ is an independently even drawing, an application of~\cite[Lemma 1]{SS13_block} shows that many blocks of $B$ are now zero matrices, see Figure~\ref{fig:LemmaRank21}.

    Indeed, for $i=1,2$, by the choice of $T$, if $e\in E_i$ and $f\in E_{3-i}\cup F_{3-i}$ then $T\cup \{e,f\}$ contains a pair of vertex disjoint cycles $C_e$ and $C_f$
    containing $e$ and $f$, respectively. By~\cite[Lemma 1]{SS13_block},  $y_{e}^{\top}y_{f} \mod 2=0$.
    Similarly, if  $e\in E_{i} \cup F_i$ and $f\in  H_{3-i}$, then $T\cup \{e,f,vu_{3-i}\}$
    contains a pair of vertex disjoint cycles $C_e$ and $C_f$
    containing $e$ and $f$, respectively. By~\cite[Lemma 1]{SS13_block}, $y_{e}^{\top}(y_{f}+y_{vu_2}) \mod 2=0$.

    In order to construct $C$, we first obtain $\begin{pmatrix}B_{11}' & B_{12}'\end{pmatrix}$ from $\begin{pmatrix}
        B_{11} & B_{12} \end{pmatrix}$ by elementary row operations, where
    $B_{11}'$ and $B_{12}'$ is obtained from $B_{11}$ and $B_{12}$, respectively, so that
    $B_{11}'=\begin{pmatrix}
        0 \\ C_{21}'
      \end{pmatrix}$, where  $C_{21}'$ has full row rank.
    Note that by performing the row operations on $\begin{pmatrix}B_{11} & \ldots & B_{16}\end{pmatrix}$ rather than just
    $\begin{pmatrix}B_{11} & B_{12}\end{pmatrix}$, we do not affect the  blocks $B_{14},B_{15}$ and $B_{16}$, since they are zero matrices.
    The remaining blocks in the first row are then refined accordingly into $\begin{pmatrix}C_{1j}' \\ C_{2j}'\end{pmatrix}=B_{1j}'$.
    We repeat the same procedure with $\begin{pmatrix}
        C_{12}' & C_{13}' \end{pmatrix}$, and vertically flipped also with $\begin{pmatrix}B_{31} & B_{32}\end{pmatrix}$, thereby obtaining the first six rows of blocks of $C$ as shown in  Figure~\ref{fig:LemmaRank21}.

    By construction, we have \\  $\mathrm{rank}\begin{pmatrix} A_{E_1,E_1} & A_{E_1,F_1} & A_{E_1,H_1}
      \end{pmatrix}=  \mathrm{rank}\begin{pmatrix}
        C_{31} \\ C_{41} \\ C_{51} \end{pmatrix}=\mathrm{rank}(C_{31})+\mathrm{rank}(C_{22})+\mathrm{rank}(C_{13})$, \\
    $ \mathrm{rank} \begin{pmatrix}
        A_{E_1,E_1} & A_{E_1,F_1} \\
        A_{H_1,E_1} & A_{H_1,F_1}
      \end{pmatrix}= \mathrm{rank}\begin{pmatrix}
        C_{31} \\ C_{51}  \end{pmatrix}+ \mathrm{rank}\begin{pmatrix}
        C_{22} \\ C_{62} \end{pmatrix},
      \begin{pmatrix} A_{E_1,E_1} & A_{E_1,F_1}
      \end{pmatrix}= \mathrm{rank}(C_{31})+\mathrm{rank}(C_{22})$, and
    $\begin{pmatrix} A_{E_1,E_1} & A_{E_1,H_1}
      \end{pmatrix}= \mathrm{rank}\begin{pmatrix}
        C_{31} \\ C_{51}
      \end{pmatrix}$.

    Thus, by~(\ref{eqn:Lemma62})
    \begin{eqnarray}
      \min_{X_2,X_3} \mathrm{rank}(B_1(X_2,X_3))&\le &  \mathrm{rank}\begin{pmatrix}
        C_{31} \\ C_{41} \\ C_{51} \end{pmatrix}+ \mathrm{rank}\begin{pmatrix}
        C_{22} \\ C_{62} \end{pmatrix} +  \mathrm{rank}\begin{pmatrix}
        C_{13} \end{pmatrix}
      \label{eqn:blockRank21}
    \end{eqnarray}

    Similarly, we obtain the last six rows of blocks of $C$.
    By symmetry and due to the deletion of the row and column corresponding to $vu_2$ we have
    \begin{eqnarray}
      \min_{X_1,X_2} \mathrm{rank}(B_2(X_1,X_2)) &\le& 3+ \mathrm{rank}\begin{pmatrix}
        C_{10,6} \\ C_{96} \\ C_{86} \end{pmatrix}+ \mathrm{rank}\begin{pmatrix}
        C_{11,5} \\ C_{75} \end{pmatrix} +  \mathrm{rank}\begin{pmatrix}
        C_{12,4} \end{pmatrix}
      \label{eqn:blockRank22}
    \end{eqnarray}

    \begin{figure}
      \centering
      $\begin{pmatrix}
          0      & 0      & C_{13} & 0        & 0        & 0        \\
          0      & C_{22} & C_{23} & 0        & 0        & 0        \\
          C_{31} & C_{32} & C_{33} & 0        & 0        & 0        \\
          C_{41} & C_{42} & C_{43} & 0        & C_{45}   & 0        \\
          C_{51} & C_{52} & C_{53} & C_{54}   & 0        & 0        \\
          0      & C_{62} & C_{63} & C_{64}   & 0        & 0        \\
          0      & 0      & C_{73} & C_{74}   & C_{75}   & 0        \\
          0      & 0      & C_{83} & C_{84}   & C_{85}   & C_{86}   \\
          0      & C_{92} & 0      & C_{94}   & C_{95}   & C_{96}   \\
          0      & 0      & 0      & C_{10,4} & C_{10,5} & C_{10,6} \\
          0      & 0      & 0      & C_{11,4} & C_{11,5} & 0        \\
          0      & 0      & 0      & C_{12,4} & 0        & 0        \\
        \end{pmatrix}\sim
        \begin{pmatrix}
          C_{13} & 0        & 0      & 0        & 0      & 0        \\
          0      & C_{12,4} & 0      & 0        & 0      & 0        \\
          C_{23} & 0        & C_{22} & 0        & 0      & 0        \\
          C_{63} & C_{64}   & C_{62} & 0        & 0      & 0        \\
          0      & C_{11,4} & 0      & C_{11,5} & 0      & 0        \\
          C_{73} & C_{74}   & 0      & C_{75}   & 0      & 0        \\
          C_{33} & 0        & C_{32} & 0        & C_{31} & 0        \\
          C_{43} & 0        & C_{42} & C_{45}   & C_{41} & 0        \\
          C_{53} & C_{54}   & C_{52} & 0        & C_{51} & 0        \\
          0      & C_{10,4} & 0      & C_{10,5} & 0      & C_{10,6} \\
          0      & C_{94}   & C_{92} & C_{95}   & 0      & C_{96}   \\
          C_{83} & C_{84}   & 0      & C_{85}   & 0      & C_{86}   \\
        \end{pmatrix}$

      \caption{Permuting the rows and columns in $C$.}
      \label{fig:LemmaRank22}
    \end{figure}

    Finally, we permute the rows and columns of blocks in $C$ by applying  the permutation $\pi_r:[12]\rightarrow [12]$ and $\pi_c: [6] \rightarrow [6]$, respectively, see Figure~\ref{fig:LemmaRank22}.
    We specify permutations as vectors $(\pi_r(1),\ldots,\pi_r(12))=(1,12,2,6,11,7,3,4,5,10,9,8)$ and $(\pi_c(1),\ldots,\pi_c(6))=(3,4,2,5,1,6)$.
    It is a routine to check that $C_{13},C_{12,4},\begin{pmatrix}
        C_{22} \\ C_{62} \end{pmatrix}, \begin{pmatrix}
        C_{11,5} \\ C_{75} \end{pmatrix}, \begin{pmatrix}
        C_{31} \\ C_{41} \\ C_{51} \end{pmatrix}$, and  $\begin{pmatrix}
        C_{10,6} \\ C_{96} \\ C_{86} \end{pmatrix}$ form a diagonal
    in this order after permuting the rows and columns of blocks and that all the above diagonal blocks are zero matrices. Hence, summing up the equalities~(\ref{eqn:blockRank21}) and~(\ref{eqn:blockRank22}), and observing that the sum of the ranks of the diagonal blocks is a lower bound on $\mathrm{rank}(B)$ which concludes the proof.
  \end{proof}

  We are done by the following two chains of inequalities.
  \begin{align*}
       -2+2\mathrm{g}_0(G) & \mathrel{\substack{(\ref{eqn:easy1amalgamation2}) \\\le\\\\ }} 2\mathrm{g}_0(G_1) + 2\mathrm{g}_0(G_2) \\
      & \mathrel{\substack{(\ref{eqn:blockBound2})                            \\\le\\\\ }}  \min_{X_2,X_3} \mathrm{rank}(B_1(X_2,X_3))+  \min_{X_1,X_2} \mathrm{rank}(B_2(X_1,X_2))+4 \\
    & \le \mathrm{rank}(B)+7 \\ & \mathrel{\substack{(\ref{eqn:egVsRank21})\\ \le \\\\ }}  2\mathrm{g}_0(G)+7,
  \end{align*}

  \begin{align*}
       -2+\mathrm{eg}_0(G) & \mathrel{\substack{(\ref{eqn:easy1amalgamation2}) \\ \le \\\\}} \mathrm{eg}_0(G_1) + \mathrm{eg}_0(G_2) \\
      & \mathrel{\substack{(\ref{eqn:blockBound2})                            \\ \le \\\\}} \min_{X_2,X_3} \mathrm{rank}(B_1(X_2,X_3)) + \min_{X_1,X_2} \mathrm{rank}(B_2(X_1,X_2)) \\
    & \le \mathrm{rank}(B)+3 \\   & \mathrel{\substack{(\ref{eqn:egVsRank21})\\ \le \\\\}}  \mathrm{eg}_0(G)+3.
  \end{align*}

  If $G$ contains the edge $uv$,
  we subdivide $uv$ by a vertex which does not affect $\mathrm{g}(G),\mathrm{g}_0(G)$,  $\mathrm{eg}(G)$ or $\mathrm{eg}_0(G)$.
  If $G-u-v$ contains more than 2 connected components, that is $k>2$,
  we add $k-2$ edges to $G$ in order to turn it into a graph $G'$ such that $G'-u-v$  has precisely 2 connected components and $G'=\Pi_{u,v}(G_1',G_2')$, where $G_1'$ and $G_2'$ is a supergraph of $G_1$ and $G_2$, respectively. Adding an edge can increase the $\mathbb{Z}_2$-genus and the Euler $\mathbb{Z}_2$-genus by at most 1 and 2, respectively.
  Thus, by the previous case $\mathrm{g}_0(G_1)+\mathrm{g}_0(G_2)\le \mathrm{g}_0(G_1')+\mathrm{g}_0(G_2')\le \mathrm{g}_0(G')+7/2 \le \mathrm{g}_0(G)+7/2+(k-2)$, and
  $\mathrm{eg}_0(G_1)+\mathrm{eg}_0(G_2)\le \mathrm{eg}_0(G_1')+\mathrm{eg}_0(G_2')\le \mathrm{eg}_0(G')+3 \le \mathrm{eg}_0(G)+2(k-2)+3$.
  \begin{remark}
    The application of Proposition~\ref{prop_strenghtening}, which can radically change the crosscap vectors (compared to the original drawing) in order to put a good upper bound on the (Euler) $\mathbb{Z}_2$-genus of the involved graphs, seems to be crucial.
    The approach of Schaefer and \v{S}tefankovi\v{c} that changes the crosscap vectors only to ensure the validity of their equation~\cite[Equation (11)]{SS13_block},
    does not seem to be extendable to the setting of 2-amalgamations as witnessed by the following example.
    Let $Z_i\subset \mathbb{Z}_2^g$ and $W_i\subset \mathbb{Z}_2^g$ be the vector space generated by the crosscap vectors of the edges of $F_i$ and $H_i$, respectively. We assume that the subspace generated by the rest of the edges is 0.
    Let both $W_1$ and $Z_2$ be generated by  $\{(1,1,0,...,0), (0,0,1,1,0, \ldots, 0), \ldots, (0,\ldots,0,1,1)\}$.
    Let $Z_1$ be generated by \\ $\{(0,1,\ldots,0),(0,0,0,1,\ldots,0),\ldots,(0,\ldots,0,1)\}$.
    Similarly, we let $W_2$ be generated by \\ $\{(1,0,\ldots,0),(0,0,1,0,\ldots, 0),\ldots, (0,\ldots,1,0)\}$.
    We have that $W_1$ is orthogonal to $Z_2$ and similarly $Z_1$ is orthogonal to $W_2$,
    and also $Z_1+W_1=Z_2+W_2=\mathbb Z_2^g$, whereas we would like to prove that
    $\mathrm{dim}(Z_1+W_1)+ \mathrm{dim}(Z_2+W_2)$ is not much bigger than $g$.
  \end{remark}
\end{proof}

\section{Conclusion}
\label{sec:conclusion}

Theorem~\ref{thm:Knm} does not determine the (Euler) $\mathbb{Z}_2$-genus of $K_{m,n}$ precisely for $m\ge 4$, and we find the problem of computing the precise values interesting already for $m=4$.
We also leave as an open problem whether in Theorem~\ref{thm:2amalg}, the dependence of the upper bounds on $k$ can be removed.

Let $G$ be a $k$-amalgamation of $G_1$ and $G_2$ for some $k\ge 3$.
On the one hand, the result of  Miller~\cite{Mi87_additivity} and Richter~\cite{R87_2amalgamation} was extended by Archdeacon~\cite{A86_add} to $k$-amalgamations, for $k\ge 3$, with the
error term $\mathrm{eg}(G_1)+\mathrm{eg}(G_2)-\mathrm{eg}(G)$ being  at most quadratic in $k$.
On the other hand, in a follow-up paper~\cite{A86_nonadd}  Archdeacon showed that for $k\ge 3$, the genus of a graph is not additive over $k$-amalgamations, in a very strong sense. In particular, the value of $\mathrm{g}(G_1)+\mathrm{g}(G_2)-\mathrm{g}(G)$ can be arbitrarily large even for $k=3$.
We wonder if the $\mathbb{Z}_2$-genus and the Euler $\mathbb{Z}_2$-genus behave in a similar way.





\bibliographystyle{plain}
\bibliography{references}

\end{document}